\documentclass[11pt]{amsart}
\usepackage[utf8]{inputenc}

\usepackage{amssymb,amsthm,amsmath}
\usepackage{mathrsfs}
\usepackage{enumerate}
\usepackage{graphicx,xcolor}
\usepackage[hidelinks]{hyperref}
\usepackage{verbatim}
\usepackage{geometry}
\usepackage{tikz}
\usetikzlibrary{calc}
\usetikzlibrary{decorations.markings}

\def\centerarcarrow[#1](#2)(#3:#4:#5)% Syntax: [draw options] (center) (initial angle:final angle:radius)
    { \draw[#1,postaction={decorate}] ($(#2)+({#5*cos(#3)},{#5*sin(#3)})$) arc (#3:#4:#5); }

\newcommand{\mb}[1]{\mathbb{{#1}}}
\newcommand{\mc}[1]{\mathcal{{#1}}}
\newcommand{\dd}{\mathrm{d}}
\newcommand{\E}{\mathbb{E}}

\newcommand{\1}{\textbf{1}}
\newcommand{\R}{\mathbb{R}}

\newcommand{\e}{\varepsilon}
\newcommand{\p}[1]{\mathbb{P}\left( #1 \right)}
\newcommand{\scal}[2]{\left\langle #1, #2 \right\rangle}

\DeclareMathOperator{\vol}{vol}
\DeclareMathOperator{\inter}{int}
\DeclareMathOperator{\supp}{supp}

\usepackage{xpatch}
\makeatletter
\def\thm@space@setup{%
  \thm@preskip=12pt plus 0pt minus 0pt
  \thm@postskip=0pt plus 0pt minus 0pt
}
\xpatchcmd{\proof}{6\p@\@plus6\p@\relax}{\z@skip}{}{}
\makeatother

\newgeometry{tmargin=3cm, bmargin=3cm, lmargin=3cm, rmargin=3cm}

\newtheorem{theorem}{Theorem}
\newtheorem{problem}{Problem}
\newtheorem{lemma}[theorem]{Lemma}
\newtheorem{corollary}[theorem]{Corollary}

\theoremstyle{remark}
\newtheorem{remark}[theorem]{Remark}

\theoremstyle{definition}

\title{Slicing $\ell_p$-balls reloaded: stability, planar sections in $\ell_1$}

\author{Giorgos Chasapis}
\address{(G.C. \& T.T.) Carnegie Mellon University}
\email{\{gchasapi,ttkocz\}@andrew.cmu.edu}

\author{Piotr Nayar}
\thanks{P.N. was supported by the National Science Centre, Poland, grant 2018/31/D/ST1/01355}
\address{(P.N.) University of Warsaw}
\email{nayar@mimuw.edu.pl}

\author{Tomasz Tkocz}
\thanks{TT's research supported in part by NSF grant DMS-1955175.}
\date{\today}

\begin{document}

\begin{abstract} 
We show that the two-dimensional minimum-volume central section of the $n$-dimensional cross-polytope is attained by the regular $2n$-gon. We establish stability-type results for hyperplane sections of $\ell_p$-balls in all the cases where the extremisers are known. Our methods are mainly probabilistic, exploring connections between negative moments of projections of random vectors uniformly distributed on convex bodies and volume of their sections.
\end{abstract}

\maketitle

{\footnotesize

}
\bigskip

\section{Introduction}

 For $p > 0$ let  $B_p^n=\{(x_1,\ldots, x_n) \in \R^n: \sum_{i=1}^n |x_i|^p \leq 1\}$ be the unit ball in the standard $\ell_p^n$ norm. The problem of determining $k$-dimensional sections of $B_p^n$ of maximal and minimal volume  proved to be notoriously difficult and 
%is a long-standing open problem in convex geometry. 
has attracted significant attention over the past few decades, notably prompting development of several important analytic, geometric and probabilistic techniques. It originated in the context of the sections of the cube from questions in geometry of numbers (see, e.g. \cite{H79, V79}).

Conspicuously, Fourier analytic methods have played a prominent role in these developments, starting perhaps with Ball's solution \cite{B86} to maximal volume hyperplane sections of the cube, and significantly advanced in the many works that followed. We refer to Koldobsky's monograph \cite{K}. In its comprehensive introduction we find the following elementary formula 
\begin{equation}\label{eq:kol}
\vol_{n-1}(K \cap a^\perp) = \frac12\lim_{\e \to 0+} \e\int_K |\scal{x}{a}|^{-1+\e} \dd x
\end{equation}
for the volume of the section of an origin-symmetric star body $K$ in $\R^n$ by the hyperplane $a^\perp$ perpendicular to a unit vector $a$ in $\R^n$. This formula can perhaps be traced back to Kalton and Koldobsky's paper \cite{KalKol}, where it appears in the context of embeddings into $L_p$-spaces with negative $p$ and the connection to intersection bodies (significant in the full resolution of the famous Busemann-Petty problem, see \cite{GKS, Lu, Zh}). 

This formula can be seen as a starting point and inspiration of the present paper. Probabilistically, the right hand side of \eqref{eq:kol}, after normalising, is the limit of the negative moments $\E|\scal{X}{a}|^{-1+\e}$ of the marginal $\scal{X}{a}$ of a random vector $X$ uniformly distributed on $K$. Since plainly $\frac{\e}{2}\int_{\R} |t|^{-1+\e}f(t)\dd t \to f(0)$ as $\e \to 0+$ for a (say bounded and continuous) density $f$ on $\R$, we get the left hand side. This point of view naturally connects the problem of extremal volume sections of convex bodies with Khinchin-type inequalities for negative moments (for the latter, in the context of the cube, we refer to the recent work \cite{CKT}). Here we employ the same idea to sharpen all the known results for extremal volume hyperplane sections of $\ell_p$-balls.

\subsection*{Notation}
We try to follow standard notation used in probability and convex geometry. For convenience we try to recall or introduce it as we move along but we also summarise most of it here. By a convex body $K$ in $\mathbb{R}^n$ we mean a compact convex set with non-empty interior. We denote by $\vol_n(A)$ the $n$-dimensional Lebesgue measure of a measurable set $A$ in $\mathbb{R}^n$, whereas $\vol_H$ will stand for the Lebesgue $k$-dimensional measure on a $k$-dimensional subspace $H$ of $\R^n$ (instead of writing $\vol_H$ we shall often write $\vol_k$, where $k$ is the dimension of $H$, if it is clear what $H$ is in a given context). For a vector $x=(x_1,\dots, x_n)$ in $\R^n$, $|x| = (\sum_{j=1}^n x_j^2)^{1/2}$ denotes its Euclidean norm, $\scal{x}{y} = \sum_{j=1}^n x_jy_j$ is the standard inner product of two vectors $x$ and $y$ in $\R^n$ and, as usual, $(e_j)_{1 \leq j \leq n}$ is the standard basis of $\R^n$, thus $\scal{e_j}{e_k} = \delta_{jk}$. The orthogonal complement of a subspace $H$ in $\R^n$ is denoted by $H^\perp$ and for a vector $a$ in $\R^n$, $a^\perp = \{x \in \R^n, \scal{x}{a}=0\}$ is the hyperplane with normal $a$. For $p > 0$, $B_p^n=\{x \in \R^n: \sum_{i=1}^n |x_i|^p \leq 1\}$ is the unit $\ell_p$-ball. In particular, $B_2^n$ is the unit Euclidean ball and its boundary, the $(n-1)$-dimensional unit sphere is denoted by $S^{n-1} = \partial B_2^n = \{x \in \R^n, |x| = 1\}$. When $p=\infty$, $B_\infty^n = [-1,1]^n$ is the $n$-dimensional unit cube and its dilate of volume $1$ is denoted by $Q_n = \frac{1}{2}B_\infty^n = [-\frac12,\frac12]^n$. The Minkowski functional (gauge function) associated with a convex body $K$ will be denoted by $\| \cdot \|_K$.

\subsection*{Our results}
It remains an open problem to determine $k$-dimensional sections of $B_p^n$ of extremal volume: the minimal ones when  $2 \leq k \leq n-2$, $0 < p < 2$, and maximal ones when  $2 \leq k \leq n-1$, $2 < p < \infty$. This paper is twofold. First, we take on this question in the case of the cross-polytope and two-dimensional sections, so for $p=1$ and $k=2$. Second, we establish stability-type results for the hyperplane sections in all of the cases where the extremisers are known. Our bounds on deficits are sharp modulo multiplicative constants.

\subsubsection*{Cross-polytope}

Our first main result is the following theorem about minimal volume two-dimensional central sections of the cross-polytope $B_1^n$. 

\begin{theorem}\label{thm:p=1}
Let   $n \geq 3$. For every $2$-dimensional subspace $H$ of $\R^n$ one has
\[
	 \vol_2(B_1^n \cap H) \geq \frac{n^2\sin^3\left(\frac{\pi}{2n}  \right)}{\cos\left( \frac{\pi}{2n} \right)}.
\]
Moreover, if the equality holds, then $B_1^n \cap H$ is isometric to a regular $2n$-gon in $\R^2$. The minimum is achieved for $H=T(\R^2)$, with $Tx=(\scal{v_1}{x}, \ldots, \scal{v_n}{x})$ and $v_k = (\cos(\frac{k\pi}{n}),\sin(\frac{k\pi}{n}))$, $k=1,\ldots,n$. The minimising subspace $H$ is unique, up to coordinate reflections and permutations.
\end{theorem}

%{\red Note that when $n=2$, this theorem continues to hold, but is vacuous: necessarily $H = \R^2$, thus $\vol_2(B_1^2 \cap H) = \vol_2(B_1^2) = 2$ which is equal to the right hand side of the stated inequality. }

In essence, the argument relies on convexity of certain functions which arise from the radial function of a planar embedding of the cross-section $B_1^n \cap H$, after leveraging the fact that it is a polygon and breaking it up into triangles.

\subsubsection*{Stability}

Our second main result concerns dimension-free refinements of the known results for hyperplane sections, providing  sharp stability of the unique extremising hyperplanes.

\begin{theorem}\label{thm:stab}
There is a positive constant $c_p$ which depends only on $p$ such that for every $n\geq 1$ and every unit vector $a = (a_1, \dots, a_n)$ in $\R^n$ with $a_1 \geq a_2 \geq \dots \geq a_n \geq 0$, we have
\begin{align}
\label{eq:st-Bp-p<2-max}
\frac{\vol_{n-1}(B_p^n \cap a^\perp)}{\vol_{n-1}(B_p^{n}\cap e_1^\perp)} &\leq \left(a_1^p+(1-a_1^2)^{p/2}\right)^{-1/p}, \qquad 0 < p < 2,\\
\label{eq:st-Bp-p<2-min}
\frac{\vol_{n-1}(B_p^n \cap a^\perp)}{\vol_{n-1}(B_p^n \cap (\frac{e_1+\dots+e_n}{\sqrt{n}})^\perp)} &\geq 1 + c_p\sum_{j=1}^n (a_j^2-1/n)^2, \qquad 0 < p < 2,\\
\label{eq:st-Bp-p>2}
\frac{\vol_{n-1}(B_p^n \cap a^\perp)}{\vol_{n-1}(B_p^{n}\cap e_1^\perp)} &\geq 1  +  c_p|a - e_1|^2, \qquad 2 < p \leq \infty,\\
\label{eq:st-Q-max}
\frac{\vol_{n-1}(B_\infty^n \cap a^\perp)}{\vol_{n-1}(B_\infty^{n}\cap (\frac{e_1+e_2}{\sqrt{2}})^\perp)} &\leq 1 - c_\infty\left|a - \frac{e_1+e_2}{\sqrt{2}}\right|.
\end{align}
 Moreover, the dependence on the right hand side of each of these inequalities on the deficit quantity $\delta = \delta(a)$ is best possible, modulo the value of constants $c_p$.
\end{theorem}

The common starting and main point of the proof of each of these results is an exact formula for $\vol_{n-1}(B_p^n \cap a^\perp)$ in terms of negative moments, as hinted in \eqref{eq:kol}. Another crucial feature common to all the proofs is that even though a random vector uniform on $B_p^n$ has dependent coordinates (except of course the cube case $p = \infty$), the dependence is \emph{mild} and the multiplicativity properties of the power function allow to replace $\scal{X}{a} = \sum a_jX_j$ in \eqref{eq:kol} with a weighted sum of \emph{i.i.d.} random variables, thanks to the well-known probabilistic representation of the uniform measure on $B_p^n$ balls in terms of the product measure with density proportional to $e^{-\sum |x_j|^p}$, see e.g. \cite{BGMN05}. The specific details of further arguments differ however, for instance as a result of the different nature of the extremising hyperplanes and resulting sections, among other things; see Section \ref{sec:heur} for an overview.

Sharpness of these results is explained in detail in the sections devoted to their proofs. %Briefly, the dependence on the right hand side of each of these inequalities on the deficit quantity $\delta = \delta(a)$ (which measures \emph{how far} $a$ is from the extremiser) is best possible, modulo the value of constants $c_p$. 

In a recent independent work \cite{MR}, Melbourne and Roberto have addressed the stability of maximal hyperplane sections of the cube, obtaining a similar result to \eqref{eq:st-Q-max}, with explicit values of the numerical constants involved. Their approach is somewhat different and relies on developing a stability version of Ball's integral inequality.

For the sake of simplicity of our arguments, we have not made any attempts to optimise the values of the involved multiplicative constants $c_p$ (or for that matter even explicitly compute some values, except for the case of \eqref{eq:st-Bp-p>2} when $p = \infty$).

\subsection*{Organisation}
We begin in Section \ref{sec:known-res} with a short overview of the relevant known results spanning the last several decades. Our new result for the cross-polytope, Theorem \ref{thm:p=1}, is proved in Section \ref{sec:p=1,k=2}.   Section \ref{sec:form} is devoted to developing the probabilistic viewpoint on sections via negative moments which forms the backbone of the proofs of our stability results from Theorem \ref{thm:stab}. 
%We recall therein well-known and develop some new formulae for volumes of sections.} 
These results are then proved in Sections \ref{sec:Q} and \ref{sec:0pinf},  preceded with some heuristics gathered in Section \ref{sec:heur}. First, we deal with the cube and prove \eqref{eq:st-Bp-p>2} for $p=\infty$ in Section \ref{sec:Qmin}, as well as \eqref{eq:st-Q-max} in Section \ref{sec:Qmax}. Then, we consider the case $0 < p < 2$ and show \eqref{eq:st-Bp-p<2-max} in Section \ref{sec:p<2max}, followed by the proof of \eqref{eq:st-Bp-p<2-min} in Section \ref{sec:p<2min}. Finally, we present the proof of \eqref{eq:st-Bp-p>2} when $2 < p < \infty$ in Section \ref{sec:p>2}. We gather some concluding comments and possible future directions in Section~\ref{sec:conclusion}.

\subsection*{Acknowledgements}

We would like to thank Fedor Nazarov for helpful discussions and for sharing with us his proof of Theorem 1 as well as letting us include it in this paper. We are also indebted to the anonymous referee for many valuable comments which helped significantly improve the manuscript.

\section{Background: known results}\label{sec:known-res}

We begin by briefly recalling the known results.
Let $H_k$ be the hyperplane perpendicular to $e_1+\ldots+e_k$, where $(e_j)_{1 \leq j \leq n}$ is the standard basis of $\R^n$. The smallest hyperplane section of the cube $B_\infty^n$ is obtained by taking the hyperplane $H_1$, which was proved by Hadwiger in \cite{H72} and independently by Hensley in  \cite{H79}. This has been generalised to sections of arbitrary dimension by Vaaler in \cite{V79}. In \cite{B86} Ball showed that $H_2$ gives the hyperplane section of the cube with the largest volume, see also \cite{NP00} for a simpler proof. This important result led to the negative answer to the Busemann-Petty question in large dimensions, see \cite{B87}. The article \cite{B89} contains a study of maximal lower dimensional sections of the cube (the results are optimal if the dimension $k$ of the subspace divides $n$ or $k \geq n/2$). It is shown in \cite{Ole} that $H_2$ is not a maximising subspace for the volume of hyperplane sections of $B_p^n$ for $p \leq 24$. 
For a comprehensive survey of the results for the cube, we refer to Chapter~1 of \cite{Z}. For some  recent related results, we also refer to \cite{Al, Am, ALM, IT, Ko, KoKo, KR20, LPP, Po}.

Meyer and Pajor studied in \cite{MP88} the same problem for $B_p^n$ with finite $p$. They showed that for any dimension $k$, the set $B_p^k$ obtained by taking the standard coordinate subspace $\mathrm{span}\{e_1,\ldots,e_k\}$ is the maximal section for $1 \leq p \leq 2$ and the minimal section for $p \geq 2$. For extensions to $p \in (0,1)$ see \cite{B95, C92}. In \cite{MP88}, Meyer and Pajor also found the minimal hyperplane section of $B_1^n$, which is given by taking the hyperplane $H_n$. Koldobsky in \cite{K98} extended this result to $p \in (0,2)$. Later on several works treated the complex case (see \cite{KolZ, OP}) as well as a further generalisation to \emph{block subspaces} (see \cite{Esk}). 
We emphasise the fact that in all of the cases, the known extremising subspaces are also known to be unique (modulo symmetries).
%In \cite{ENT18} the authors extended Koldobsky's result by showing the following Schur-concavity result for this range of $p$: if $a=(a_1,\ldots,a_n)$ and $b=(b_1,\ldots,b_n)$ are unit vectors and $(a_1^2,\ldots,a_n^2)$ is majorized by $(b_1^2,\ldots,b_n^2)$ in the Schur order, then the volume of $B_p^n \cap a^\perp$ is less than or equal to the volume of $B_p^n \cap b^\perp$. 

We mention in passing that the analogous, dual question for extremal \emph{projections} of $B_p^n$ has also been considered. The problem is related to certain Khinchin-type inequalities, as explained in \cite{Ball95,BN02}. In particular, finding extremal projections of $B_1^n$ is equivalent to deriving optimal constants in the classical Khinchin inequality, which was done by Szarek in \cite{S76}, followed up by De, Diakonikolas and Servedio who developed a stability version in \cite{DDS}. The case $p \geq 2$ has been studied by Barthe and Naor in \cite{BN02}, where the authors showed that the smallest and the largest $(n-1)$-dimensional projections of $B_p^n$ are those onto the hyperplanes $H_1$ and $H_n$, respectively. Koldobsky, Ryabogin and Zvavitch in \cite{KRZ} developed a Fourier analytic approach. Chakerian and Filliman in \cite{ChF} found that the $2$-dimensional orthogonal projections of the cube $B_\infty^n$ of maximal volume are attained by regular $2n$-gons (the same extremiser as in our Theorem \ref{thm:p=1}) and, by McMullen's formula from \cite{McM}, this also gives $(n-2)$-dimensional projections of maximal volume. See \cite{Iv2} for recent results on lower dimensional projections of the cross-polytope $B_1^n$.
%Schur concavity extension of this result was given in \cite{ENT18}.
Paper \cite{ENT18} provides a different unified probabilistic approach to the volume and mean-width of central sections and projections and in addition to identifying the extremisers, also delivers Schur-convexity-type results.

\section{Two-dimensional central sections of the cross-polytope}\label{sec:p=1,k=2}

For the proof of Theorem \ref{thm:p=1} we first need to recall the direct elementary approach to sections viewed as linear embeddings.

\subsection{Sections via linear embeddings}\label{sec:initial}

Recall that $\|\cdot\|_K$ refers to the Minkowski functional of a convex body $K$ (if $K$ is symmetric, it is the norm whose unit ball is $K$).  We shall use the following standard lemma.

\begin{lemma}\label{lem:1}
Let $K$ be a convex body in $\R^n$ and let $T:\R^k \to \R^n$ be a linear map. Define $K_T=\{x \in \R^k: \ \|Tx\|_K \leq 1\}$. Then $K \cap T(\R^k) = T(K_T)$. Moreover, if $\,T$ is of full rank then
\[
	\vol_{T(\R^k)}(K \cap T(\R^k)) = \sqrt{\det(T^\ast T)} \vol_k(K_T).
\]
\end{lemma}
\begin{proof}
For the first part, let us show two inclusions. If $y \in K \cap T(\R^k)$, then $y \in K$ and $y=Tx$ for some $x \in \R^k$. It follows that $\|Tx\|_K \leq 1$, so $x \in K_T$. Thus $y=Tx \in T(K_T)$. Now, if $y \in T(K_T)$, then $y=Tx$ for some $x$ satisfying $\|Tx\|_K \leq 1$. Thus $\|y\|_K \leq 1$, so $y \in K$. Since clearly $y \in T(\R^k)$, it follows that $y \in K \cap T(\R^k)$. 

For the second part, observe that one can treat $H=T(\R^k)$ as a manifold parameterised by $T$. Since $\vol_H$ is volume on this manifold, we have the well-known formula for the volume element, $\dd \vol_H = \sqrt{\det((D T)^\ast (D T)) } \ \dd\vol_k$, where $DT$ stands for the derivative of $T$. In our case $D T= T$ and so the assertion follows.      
\end{proof}

A straightforward application of the above lemma to the case of $K$ being the $B_p^n$ ball yields the following corollary.

\begin{corollary}\label{cor:1}
Suppose that $H$ is an image of $\R^k$ under a linear map $T:\R^k \to \R^n$ of full rank, given by $Tx=(\scal{v_1}{x}, \ldots, \scal{v_n}{x})$ for some vectors $v_1,\ldots, v_n \in \R^k$. Then 
\[
	\vol_H(B_p^n \cap H) = \det\left( \sum_{i=1}^n v_i \otimes v_i \right)^{1/2} \vol_k\left( \left\{ x \in \R^k: \ \sum_{i=1}^n |\scal{v_i}{x}|^p \leq 1 \right\} \right).
\]
\end{corollary}

%\begin{comment}
%Let us now observe that a subspace $H$ being an image of $\R^k$ under a linear map $T:\R^k \to \R^n$ does not change if we replace $T$ with $T \circ A$ for some linear invertible map $A:\R^k \to \R^k$. Observe that $T$ being of full rank (i.e. having trivial kernel) is equivalent to $T^\ast T$ being invertible (and thus symmetric positive definite), as can be seen from the identity $\scal{T^\ast T x}{x}=\scal{Tx}{Tx}=|Tx|^2 \geq 0$. So, if $T$ is of full rank, there exists an orthogonal map $U:\R^k \to \R^k$ such that $U^\ast T^\ast T U = D$, where $D$ is diagonal with positive entries. Therefore $(TA)^\ast TA = I_{k \times k}$, where $A=UD^{-1/2}$ is invertible. Considering the map $\tilde{T}=T \circ A$ instead of $T$ does not change $H$. Moreover, for $\tilde{T}$ we have $\tilde{T}^\ast \tilde{T}=I_{k \times k}$. Thus, finding extremal  $k$ dimensional sections of $K$ is equivalent to solving the following problem.
%\end{comment}

Here, as usual, $v \otimes v$ is the matrix $vv^\top$. Let us now assume that the map $T$ is an isometric embedding. This means that $\scal{x}{y}=\scal{Tx}{Ty}=\scal{x}{T^\ast T y}$, which gives the condition $T^\ast T = I_{k \times k}$, where $I_{k \times k}$ stands for the $k \times k$ identity matrix. If the mapping is written in the form $Tx=(\scal{v_1}{x}, \ldots, \scal{v_n}{x})$,  the condition $T^\ast T = I_{k \times k}$ rewrites as $\sum_{i=1}^n v_i \otimes v_i = I_{k \times k}$. Thus, finding extremal  $k$ dimensional sections of $K$ is equivalent to solving the following problem.

%Indeed, the map $T^\ast T:\R^k \to \R^k$ cannot be singular since $T$ being of full rank is equivalent to $T^\ast T$ being invertible, as can be seen from the identity $\scal{T^\ast T x}{x}=\scal{Tx}{Tx}=|Tx|^2$.   

\begin{problem}\label{prob:1}
Maximise/minimise the volume of the set $K_T=\{x \in \R^k: \ \|Tx\|_K \leq 1\}$ under the constrain $T^\ast T = I_{k \times k}$. In the case of $K=B_p^n$, maximise/minimise the volume of the set 
\[
K_v = \left\{ x \in \R^k: \ \sum_{i=1}^n |\scal{v_i}{x}|^p \leq 1 \right\} \quad  \textrm{over} \quad v_1,\ldots,v_n \in \R^k, \ \sum_{i=1}^n v_i \otimes v_i = I_{k \times k}.
\] 
\end{problem}

\begin{remark}\label{rem:isometry}
Since the condition $T^\ast T=I_{k \times k}$ ensures that the map is an isometric embedding, the set $K_T$ in $\R^k$ in the above extremization problem is isometric to the section $K \cap T(\R^k)$.    
\end{remark}

\subsection{Proof of Theorem \ref{thm:p=1}}\label{sec:proof-p=1}

This proof was kindly communicated to us by Fedor Nazarov.
Recall that our goal is to minimise the volume of the set $K_v=\{x \in \R^2: \ \sum_{i=1}^n |\scal{v_i}{x}| \leq 1\}$ under the constraint $\sum_{i=1}^n v_i \otimes v_i=I_{2 \times 2}$. In general, the set $K_v$ is a convex symmetric $2k$-gon, $k\leq n$. We point out that some of the vectors $v_i$ might be zero, and some of them may be parallel.  While studying the geometry of $K_v$, one can assume that the vectors $v_i$ are non-parallel, since if for some $a_1,\ldots, a_l$, $i_1, \ldots, i_l$ and $v$ one has $v_{i_1}=a_1 v, \ldots,  v_{i_l}=a_l v$, then considering only one vector $\tilde{v}=\sum_{j=1}^l |a_{i_j}|v$ instead of the vectors $v_{i_j}$ will result in the same set. However, this operation in general affects the constraint $\sum_{i=1}^n v_i \otimes v_i=I_{2 \times 2}$.  

Let $\rho:S^1 \to (0,\infty)$, given by $\rho(\theta)=\left(\sum_{i=1}^n |\scal{v_i}{\theta}|\right)^{-1}$, be the radial function of $K_v$. One can assume that in our configuration there are at least two non-parallel vectors (otherwise the resulting set is an infinite strip and so its volume is infinite; in this case $\sum_{i=1}^n v_i \otimes v_i$ is of rank one, and the constraint is not satisfied). It is not hard to check that under this assumption the vertices of $K_v$ correspond exactly to directions $\theta$ perpendicular to $v_i$ for some non-zero $v_i$ (that is,  up to the changes of sign of $\scal{v_i}{\theta}$). Indeed, for points $x$ on the boundary of $K_v$ one has $\sum_{i=1}^n |\scal{v_i}{x}|=1$. If in a small neighborhood of $x$ all the signs of $\scal{v_i}{x}$ are fixed, this is a linear equation and the set of solutions is a line  which corresponds to $1$-dimensional faces of $K_v$. If on the other hand $x$ satisfies $\scal{v_i}{x}=0$ for some non-zero $v_i=(a,b)$ (if there are vectors parallel to $v_i$ we join them together as above), then within a small ball around $x=(s_0,t_0)$ there is a part of the boundary being a subset of the line of the form $\{(s,t): as+bt+As+Bt=1\}$ and a part being a subset of the line of the form $\{(s,t): -as-bt+As+Bt=1\}$. %Note that $(A,B) \ne (0,0)$ since we assume that in our set of vectors there exist at least two non-parallel vectors $v_i$. 
These two lines intersect each other at $x$. We shall show that they are non-parallel. If they were parallel, they would have to coincide and thus we would have $a+A=-a+A$ and $b+B=-b+B$, which gives $a=b=0$, contradiction.
%\[
%	\det\left( \begin{array}{cc}
%	a+A & b+ B \\
%	-a+A & -b+B
%\end{array}	 \right) = 0, \qquad \det\left( \begin{array}{cc}
%	a+A & 1 \\
%	-a+A & 1
%\end{array}	 \right) = 0, 
%\qquad
%\det\left( \begin{array}{cc}
%	1 & b+ B \\
%	1 & -b+B
%\end{array}	 \right) = 0. 
%\]        
%The last two equations easily yield $a=b=0$, contradiction. 
Thus $x$ is an intersection of two non-parallel parts of the boundary and thus is a vertex of $K_v$. A simple consequence of these observations is that $K_v$ has at most $2n$ vertices.

Suppose that the boundary of $K_v$ consists of segments $F_j$, $j=1,\ldots,k$. Let $C_j$ be the corresponding segments of $S^1$, that is $\theta \in C_j$ if $\rho(\theta) \theta \in F_j$, and let $T_j=\textrm{conv}(0,F_j)$ be the corresponding triangle in $K_v$. We define $A_j = \frac12 \int_{C_j} \rho^2$ and $I_j = \int_{C_j} \rho^{-1}$. Suppose that the angle of $T_j$ at vertex $O=0$ has measure $2 \beta_j$, where $\beta_j \in (0,\pi/2)$. Note that $\sum_{j=1}^k \beta_j = \pi$. We shall need the following elementary lemma.

\begin{figure}[htb] 
  \centering
  \def\svgwidth{250pt}
  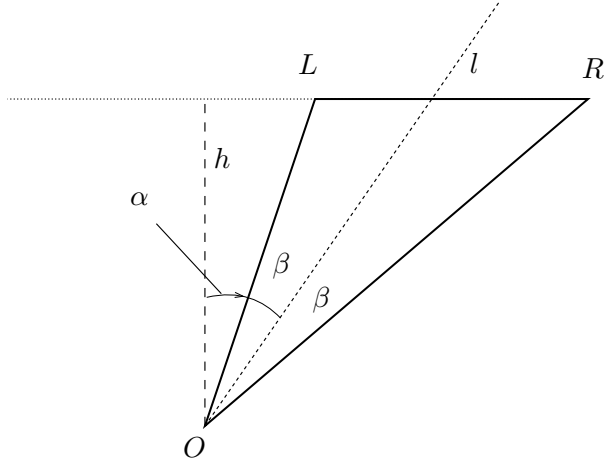
  \caption{One \emph{piece} of $K_v$: triangle $OLR$.} \label{triangle}
\end{figure}

\begin{lemma}\label{lem:1.1}
We have $A_j I_j^2 \geq \frac{4 \sin^3 \beta_j}{\cos \beta_j}$. 
\end{lemma}    
\begin{proof}
Let $OLR$ be one of our triangles $T_j$ and let $2\beta$ be the measure of the angle at vertex $O$. %By rotating we can assume that the side $LR$ is parallel to the $OX$ axis. 
Let $h$ be the height of $OLR$ perpendicular to $LR$ and let $l$ be the bisector of $\angle LOR$. The directed angle from $h$ to $l$ will be denoted by $\alpha$. Let $\theta$ be the directed angle on $S^1$, where $\theta = 0$ corresponds to points on $h$. Clearly $\rho(\theta)=h/\cos \theta$.  We have
\begin{align*}
I_j &= \int_{\alpha-\beta}^{\alpha+\beta} \frac{\cos \theta}{h} \dd \theta = \frac1h[\sin(\alpha+\beta)-\sin(\alpha-\beta)], \\
A_j &= \frac12 h^2 \int_{\alpha-\beta}^{\alpha+\beta} \frac{1}{\cos^2 \theta} \dd \theta = \frac12 h^2 [\tan(\alpha+\beta)-\tan(\alpha-\beta)].     
\end{align*}
Thus,
\begin{align*}
	A_j I_j^2 & = \frac12 \left[ \frac{\sin(\alpha+\beta)}{\cos(\alpha+\beta)} - \frac{\sin(\alpha-\beta)}{\cos(\alpha-\beta)}  \right] \cdot [\sin(\alpha+\beta)-\sin(\alpha-\beta)]^2 =  \frac{2 \sin(2 \beta) \cdot \sin^2 \beta \cos^2 \alpha}{\cos(\alpha+\beta)\cos(\alpha-\beta)} \\
	& = \frac{4 \sin^3 \beta \cos \beta \cos^2 \alpha}{\cos^2 \alpha \cos^2 \beta - \sin^2 \alpha \sin^2 \beta} =  \frac{4 \sin^3 \beta}{\cos \beta} \cdot \frac{1}{1-\tan^2 \alpha \tan^2 \beta} \geq \frac{4 \sin^3 \beta}{\cos \beta}.
\end{align*}
\end{proof}

\begin{lemma}\label{lem:2}
The function $\psi(x)= \frac{\sin x}{(\cos x)^{1/3}}$ is strictly convex on $[0,\pi/2)$. In particular, the function $[0,\pi/2) \ni x \mapsto \psi(x)/x$ is non-decreasing  and thus the sequence $a_n =  \frac{n \sin\left(\frac{\pi}{2n}  \right)}{\cos^{1/3}\left( \frac{\pi}{2n} \right)}$ is non-increasing.
\end{lemma}

\begin{proof}
Observe that $\psi'(x)=\cos^{2/3} x+ \frac13 \sin^2 x \cos^{-4/3} x = \frac23 \cos^{2/3}x + \frac13 \cos^{-4/3}x$. It suffices to show that this function is strictly increasing. Taking $y=\cos^{2/3} x$ we see that this is equivalent to showing that $f(y)= 2 y+ y^{-2}$ is strictly decreasing $(0,1)$. This is true since $f'(y)=2(1-y^{-3})<0$ for $y \in (0,1)$. 

 The second part follows from the monotonicity of the slopes of convex functions and the fact that $\psi(0)=0$.
\end{proof}

We are now ready to prove Theorem \ref{thm:p=1}.
 
\begin{proof}[Proof of Theorem \ref{thm:p=1}]
We shall solve Problem \ref{prob:1}. Assume that $\sum_{i=1}^n v_i \otimes v_i = I_{2 \times 2}$ and that $K_v$ is a convex symmetric $2k$-gon, where $k \leq n$. Note that 
\[
	\int_{S^1} \rho(\theta)^{-1} \dd \theta = \sum_{i=1}^n \int_{S^1} |\scal{v_i}{\theta}| \dd \theta = 4 \sum_{i=1}^n |v_i| \leq 4 \sqrt{n} \sqrt{\sum_{i=1}^n |v_i|^2} = 4 \sqrt{2n},
\]
where in the last equality we use $\sum_{i=1}^n |v_i|^2 = \textrm{tr}\left(\sum_{i=1}^n v_i \otimes v_i \right)$.
Moreover, using H\"older's inequality, Lemma \ref{lem:1.1} and Lemma \ref{lem:2}, we get
\begin{align*}
	|K_v|^{\frac13} (4 \sqrt{2n})^{\frac23} &  \geq |K_v|^{\frac13} \left( \int_{S^1} \rho(\theta)^{-1} \dd \theta \right)^{\frac23} = \left(\sum_{j=1}^{2k} A_j\right)^{\frac13}\left(\sum_{j=1}^{2k} I_j\right)^{\frac23} \\
	&\geq \sum_{j=1}^{2k} A_j^{\frac13} I_j^{\frac23} \geq 4^{\frac13}\sum_{j=1}^{2k} \frac{\sin \beta_j}{\cos^{1/3} \beta_j}  \\
	& \geq 4^{\frac13} \cdot 2k \frac{\sin\left(\frac{1}{2k} \sum_{j=1}^{2k} \beta_j \right)}{\cos^{1/3}\left( \frac{1}{2k}\sum_{j=1}^{2k} \beta_j \right)} = 2 \cdot 4^{\frac13}  \cdot \frac{k \sin\left(\frac{\pi}{2k}  \right)}{\cos^{1/3}\left( \frac{\pi}{2k} \right)} \geq  2 \cdot 4^{\frac13}  \cdot \frac{n\sin\left(\frac{\pi}{2n}  \right)}{\cos^{1/3}\left( \frac{\pi}{2n} \right)}. 
\end{align*}
We arrive at $|K_v| \geq \frac{n^2\sin^3\left(\frac{\pi}{2n}  \right)}{\cos\left( \frac{\pi}{2n} \right)}$.

We now show that this bound is achieved for $K_v$ being a regular $2n$-gon. Let us consider $v_k=\sqrt{\frac{2}{n}}(\cos(\frac{k \pi}{n}), \sin(\frac{k \pi}{n}))$ for $k=1, \ldots, n$. It is easy to verify that $\sum_{i=1}^n v_i \otimes v_i = I_{2 \times 2}$. As we already mentioned, the vertices of $K_v$ correspond to the directions perpendicular to $v_i$. Since $v_i$ are equally spaced on the upper half-circle, we get that $K_v$ is a regular $2n$-gon. Clearly $|v_1|= \ldots = |v_n|$, $\beta_1 = \ldots = \beta_{2n}$,  $I_1= \ldots = I_{2n}$ and $A_1 = \ldots = A_{2n}$. Thus, one has equalities in all the inequalities in the above proof, so $|K_v| = n^2\sin^3\left(\frac{\pi}{2n}  \right)/\cos\left( \frac{\pi}{2n} \right)$. Conversely, it is easy to see that the only possibility of having equalities in all the estimates of the proof is to have the set $\{v_1, -v_1, \ldots, v_n, -v_n\}$ equally spaced on the circle. Thus, in the extremal case the only freedom of choosing $v_i$ is to apply rotations to all the vectors $v_i$ (which does not change the section $B_1^n \cap T(\R^2)$, as it corresponds to replacing $T$ with $T\circ U$ for some orthogonal transformation $U$ of $\R^2$), permuting some of the vectors (which corresponds to applying permutations of coordinates  in $\R^n$, under which $H$ changes), and reflecting some of the vectors $v_i$ (which corresponds to applying coordinate reflections in $\R^n$ which again changes $H$). Thus, up to coordinate reflections and permutations, there is only one minimal two-dimensional section of $B_1^n$.  The fact that the section of minimal volume is isometric to a regular $2n$-gon in $\R^2$ follows from Remark \ref{rem:isometry}.    
\end{proof} 
 
%\begin{remark}\label{rmk:1}
%We remark that the section of $B_1^n$ of minimal volume is in fact isometric to a regular $2n$-gon in $\R^2$. Indeed, in the extremal case $K_v$ is a regular $2n$-gon and  the map $Tx=(\scal{v_1}{x},\ldots, \scal{v_n}{x})$ is an isometric embedding of $\R^2$ into $\R^n$, whenever the condition $\sum_{i=1}^n v_i \otimes v_i = I_{2 \times 2}$ is satisfied (see Remark \ref{rem:isometry}). 
%\end{remark} 

\section{ Negative moments approach}
\label{sec:form}

\subsection{Formulae for sections via negative moments} 
The goal of this section is to connect extremal-volume sections of convex bodies to sharp Khinchin-type inequalities for negative moments.
%In what follows instead of writing $\vol_H$ for a Lebesgue measure on an affine subspace $H$ we shall write $\vol_k$, where $k$ is the dimension of $H$, if it is clear what $H$ is.

\begin{lemma}\label{lm:sec-f0}
Let $X$ be random vector with density $g$ in $\R^n$. Let $H$ be a codimension $k$ subspace of $\R^n$ and let $U$ be a   $k \times n$ matrix whose rows $u_1,\ldots,u_k$ form an orthonormal basis of $H^\perp$, the orthogonal complement of $H$. Then 
$
	f(x) = \int_{H + U^\top x} g
$
%\[
%f(x) = \vol_{d-k}(A \cap (H + U^T x  ))
%\]
is the density of the random vector $UX$ in $\R^k$. 
%In particular, we have 
%\[
%\vol_{d-k}(A \cap H) = f(0).
%\]
\end{lemma}

\begin{proof}
For $x=(x_1,\ldots, x_k)$ we have $U^\top x = \sum_{i=1}^k u_i x_i$. Since $u_i$ span $H^\perp$, we get that $y \in H^\perp$ iff $y=U^\top x$ for some $x \in \R^k$. Moreover, since $u_i$ are orthonormal, we get that $x \mapsto U^\top x$ is an isometric embedding of $\R^k$ into $\R^n$, whose image is $H^\perp$. By Fubini's theorem $f$ is measurable on $\R^k$. 

Let us now take a measurable set $B \subseteq \R^k$. Note that $H=\{x \in \R^n: \scal{x}{u_i}=0, 1 \leq i \leq k \}$ and thus $H= \ker U$. Every point $y \in U^{-1}(B)$ can be written as $y=y_1+y_2$, where $y_1 \in H$ and $y_2 \in H^\perp \cap U^{-1}(B)$. Since every point in $H^\perp$ is of the form $y_2=U^\top z$ for $z \in \R^k$ and $U^\top z \in U^{-1}(B)$ iff $U U^\top z \in B$, which is just $z \in B$ as $UU^\top = I_{k\times k}$, we get that $U^{-1}(B)=H+U^\top B$. Thus, by Fubini's theorem we get
\[
\p{UX \in B} =\p{X \in U^{-1}(B)} = \p{X \in H+U^\top B}  = \int_{B} \left(\int_{H + U^\top x} g \right) \dd x  = \int_B f(x) \dd x.
\]
\end{proof}

\begin{corollary}\label{cor:sec-dens}
Let $A$ be a measurable set in $\R^n$ of volume $1$ and let $X$ be a uniform random vector on $A$. Let $H$ be a codimension $k$ subspace of $\R^n$ and let $U$ be a   $k \times n$ matrix whose rows form an orthonormal basis of $H^\perp$, the orthogonal complement of $H$. Then 
\[
f(x) = \vol_{n-k}(A \cap (H + U^\top x  ))
\]
is the density of the random vector $UX$ in $\R^k$. Moreover, if $A$ is a convex body, then on its support the above function is the unique continuous version of the density of $UX$.  This continuous version satisfies 
\[
f(0) = \vol_{n-k}(A \cap H)
\] 
if $0 \in \inter \supp(f)$.  
\end{corollary}

\begin{proof}
This is a special case of Lemma \ref{lm:sec-f0}. If $A$ is a convex body, then  by Brunn-Minkowski inequality $f^{\frac{1}{n-k}}$ is concave on the interior of its support and therefore continuous.
\end{proof}

\begin{lemma}\label{lm:f0-neg-mom}
Let $X$ be a  random vector in $\R^k$ with density $f$ such that $\|f\|_\infty = f(0)$ and $f$ is lower semi-continuous at $0$. Let $\|\cdot\|$ be a norm on $\R^k$ with closed unit ball $K$. We have,
\[
f(0) = \lim_{q \to k-} \frac{k-q}{k \cdot\vol_k(K)}\E\|X\|^{-q}.
\]
\end{lemma}
\begin{proof}
We first claim that
\begin{equation}\label{eq:vol-via-norm}
	\int_{tK}\|x\|^{-q} \dd x = \frac{k}{k-q}t^{k-q} \vol_k(K), \qquad \textrm{for} \ \ t > 0, \ 0 < q < k.
\end{equation}
Indeed, thanks to the homogeneity of volume, we have  
\begin{align*}
\int_{tK}\|x\|^{-q} \dd x & = \int_{tK} \int_{\|x\|}^\infty qs^{-(q+1)} \dd s  \dd x = \int_{tK} \left(\int_{0}^\infty qs^{-(q+1)} \1_{\|x\| \leq s} \dd s \right)  \dd x  
\\
&  = \int_{0}^\infty qs^{-(q+1)}  \left(\int_{tK} \1_{\|x\| \leq s}   \dd x \right) \dd s =  \int_{0}^\infty qs^{-(q+1)}  \left(\int_{\R^k} \1_{\|x\| \leq \min(s,t)}   \dd x \right) \dd s \\
& =   \vol_k(K) \int_{0}^\infty qs^{-(q+1)}  \min(s,t)^k \dd s = \frac{k}{k-q}t^{k-q} \vol_k(K).
\end{align*}
Take $M>0$. Using \eqref{eq:vol-via-norm} with $t=M$, we get
\begin{align*}
\frac{k-q}{k\cdot\vol_k(K)}\E\|X\|^{-q} & = \frac{k-q}{k\cdot\vol_k(K)} \int_{MK} \|x\|^{-q} f(x) \dd x + \frac{k-q}{k\cdot\vol_k(K)} \int_{(MK)^c} \|x\|^{-q} f(x) \dd x 
\\ & \leq  \frac{k-q}{k\cdot\vol_k(K)}\|f\|_\infty \int_{MK}\|x\|^{-q} \dd x + \frac{k-q}{k\cdot\vol_k(K)} M^{-q} \\
&  = \|f\|_\infty M^{k-q} + \frac{k-q}{k\cdot\vol_k(K)} M^{-q}.
\end{align*}
Fix $\e > 0$. Since $\|f\|_\infty = f(0)$ and $f$ is lower semi-continuous at $0$, the set $\{x \in \R^k, f(x) > \|f\|_\infty - \e\}$ contains a neighbourhood of $0$, say $\delta K$ for some $\delta > 0$. Then,
\begin{align*}
\frac{k-q}{k\cdot\vol_k(K)}\E\|X\|^{-q} &\geq \frac{k-q}{k\cdot\vol_k(K)} \int_{\delta K} \|x\|^{-q} f(x) \dd x \\
&\geq  \frac{k-q}{k\cdot\vol_k(K)}(\|f\|_\infty-\e) \int_{\delta K}\|x\|^{-q} \dd x \\
&= (\|f\|_\infty-\e) \delta^{k-q}.
\end{align*}
These two bounds show that as $q \to k-$, the $\liminf$ and $\limsup$ of $\frac{k-q}{k\cdot\vol_k(K)}\E\|X\|^{-q}$ are within $\e$ of $\|f\|_\infty$.
\end{proof}

Combining Corollary  \ref{cor:sec-dens} and Lemma \ref{lm:f0-neg-mom} yields a probabilistic formula for sections in terms of negative moments.

\begin{corollary}\label{cor:sec-via-moments}
Let $A$ be a symmetric convex body in $\R^n$ of volume $1$ and let $X$ be uniform on $A$.  Let $\|\cdot\|$ be a norm  in $\R^k$ with closed unit ball $K$. Let $H$ be a codimension $k$ subspace of $\R^n$ and let $U$ be a   $k \times n$ matrix whose rows form an orthonormal basis of $H^\perp$. Then
\[
	\vol_{n-k}(A \cap H) = \lim_{q \to  k-} \frac{k-q}{k \cdot\vol_k(K)}\E\|U X\|^{-q}.
\]  
\end{corollary}

\begin{proof}
Since $UX$ is log-concave and symmetric on $\R^k$, one gets $\|f\|_\infty=f(0)$.
\end{proof}

\subsection{Sections of the cube} 
As a first application, we sketch how to obtain a convenient probabilistic formula for central section of the cube in terms of negative moments. It was derived first perhaps in \cite{KoKol} and later appeared in \cite{Brz} as well as \cite{KR20}. Our argument is different, more direct, bypassing the Fourier-analytic identities involving Bessel functions. It was recently presented in full detail in \cite{CKT}. It is more convenient to treat the cube of unit volume, so we set
\[
Q_n = \frac{1}{2}B_\infty^n = \left[-\frac12,\frac12 \right]^n.
\]

\begin{lemma}[K\"onig-Koldobsky, \cite{KoKol}]\label{lm:sec-cube-codim1}
For a unit vector $a = (a_1,\ldots,a_n)$ in $\R^n$, we have
\[
\vol_{n-1}\left(Q_n \cap a^\perp \right) =  \E\left|\sum_{k=1}^n a_k\xi_k\right|^{-1},
\]
where the $\xi_k$ are uniform on $S^2$ in $\R^3$.
\end{lemma}
\begin{proof}
Let $U_1,\ldots, U_n$ be i.i.d. uniform on $[-1,1]$. From Corollary \ref{cor:sec-via-moments} applied with $k=1$ one gets
\[
\vol_{n-1}\left(Q_n \cap a^\perp \right) = \lim_{q \to 1-} (1-q)\E\left|\sum_{k=1}^n a_kU_k\right|^{-q}. 
\]
It is therefore enough to show that for $q<1$ one has
\[
\E\left|\sum_{k=1}^n a_k\xi_k\right|^{-q} = (1-q)\E\left|\sum_{k=1}^n a_kU_k\right|^{-q}.
\]
This can be shown by repeating Lata\l a's argument leveraging rotational symmetry from Proposition 4 in \cite{KK01}. It has also been written in full detail in Lemma 3 in \cite{CKT}.
\end{proof}

\begin{remark}\label{rem:Qn-sec}
The following alternative Fourier-analytic formula for the volume of central codimension $1$ sections perhaps goes back to P\'olya and is well known (see, e.g. \cite{B86})
\[
	\vol_{n-1}(Q_n \cap a^\perp)= \frac{2}{\pi}\int_0^\infty \prod_{j=1}^n \frac{\sin(a_jt)}{a_jt} \dd t.
\] 
\end{remark}

\subsection{Sections of \texorpdfstring{$B_p^n$}{Bpn} via negative moments}

Let $p > 0$. Throughout the paper, we let 
\[
Y_1^{(p)}, Y_2^{(p)}, \dots \ \ \text{be i.i.d. random variables with density $e^{-\beta_{p}^p |x|^p}$},
\] 
where 
\[
\beta_p = 2\Gamma(1+1/p)
\]
is chosen such that $\int_{\R} e^{-\beta_p^p|x|^p} \dd x = 1$. We shall derive the following lemma.

\begin{lemma}\label{lm:Bpn-sec-negative-moments}
Let $H$ be a subspace in $\R^n$ of codimension $k$ such that the rows of a $k \times n$ matrix $U$ form an orthonormal basis of $H^\perp$. Let $v_1, \dots, v_n$ be the columns of $U$. Then
\[
	\frac{\vol_{n-k}(B_p^n \cap H)}{\vol_{n-k}(B_p^{n-k})} = \lim_{q\to k-} \frac{k-q}{k\vol_k(B_2^k)}  \mb{E}\Big|\sum_{j=1}^n Y_j^{(p)} v_j \Big|^{-q}.
\]
\end{lemma}

\begin{proof}
%Let $H$ be a subspace in $\R^n$ of codimension $k$ such that the rows of an $k \times n$ matrix $U$ form an orthonormal basis of $H^\perp$. 
Let $v_1, \dots, v_n$ be the columns of $U$. Note that
\[
\sum_{j=1}^n v_jv_j^\top = I_{k \times k}.
\]
We take $X=(X_1,\ldots, X_n)$ to be uniform on $B_p^n$. Then $X/\vol_n(B_p^n)^{1/n}$ is uniform on $\tilde{B}_p^n = B_p^n/\vol_n(B_p^n)^{1/n}$, which has volume $1$. Using Corollary \ref{cor:sec-via-moments} with the Euclidean norm $|\cdot|$ gives
\[
\frac{\vol_{n-k}\left(B_p^n \cap H\right)}{(\vol_n(B_p^n))^{n-k}} = \vol_{n-k}\left(\tilde{B}_p^n \cap H\right) = \lim_{q\to k-} \frac{\vol_n(B_p^n)^{\frac{q}{n}}(k-q)}{k\vol_k(B_2^k)}  \E\left|\sum_{j=1}^n X_j v_j \right|^{-q}. 
\]
We shall now use two important facts:
\vspace{0.2cm}
\begin{itemize}\itemsep=0.2cm
\item[(a)] (Barthe, Gu\'edon, Mendelson, Naor,  \cite{BGMN05}) Let $Y_1,\ldots, Y_n$ be i.i.d. random variables with densities $\beta_p^{-1}e^{-|x|^p}$ and write $Y=(Y_1,\ldots,Y_n)$. Define $S=\big(\sum_{j=1}^n |Y_j|^p\big)^{1/p}$. Let $\mathcal{E}$ be an exponential random variable with density  $e^{-t}{\bf 1}_{\{t>0\}}$, independent of the $Y_j$. Then the random vector 
$\frac{Y}{(S^p+\mathcal{E})^{1/p}}$ 
is uniformly distributed on $B_p^n$.
\item[(b)] (Schechtman, Zinn, see \cite{SZ90} and Rachev, R\"uschendorf, \cite{RR91}) With the above notation  $S$ and $Y/S$ are independent.
\end{itemize}  
\vspace{0.2cm}
In \cite{BGMN05}  Barthe, Gu\'edon, Mendelson and Naor observed that using (a) and (b) one gets
\begin{equation*}
\mb{E} \Big| \sum_{j=1}^n X_j v_j \Big|^{-q} = \mb{E} \Big| \frac{1}{(S^p+\mathcal{E})^{1/p}} \sum_{j=1}^n Y_j v_j \Big|^{-q} = \mb{E} \Big| \frac{S}{(S^p+\mathcal{E})^{1/p}}\Big|^{-q} \mb{E} \Big|\sum_{j=1}^n  \frac{Y_j}{S} v_j\Big|^{-q}.
\end{equation*}
It follows that $\mb{E} \Big| \frac{S}{(S^p+\mathcal{E})^{1/p}}\Big|^{-q}$ is finite. Thus
\[
	e^{-1} \E|S|^{-q} = \E|S|^{-q} \1_{\mc{E}>1} \leq  \mb{E} \Big| \frac{S}{(S^p+\mathcal{E})^{1/p}}\Big|^{-q} <\infty. 
\]
Then, again by independence of $S$ and $Y/S$, we have 
\[
\mb{E}\big| \sum_{j=1}^n  \frac{Y_j}{S} v_j\big|^{-q} \mb{E}|S|^{-q} = \mb{E} \big| \sum_{j=1}^n Y_j v_j\big|^{-q}
\]
and therefore
\begin{align*}
\mb{E}\Big| \sum_{j=1}^n  X_j v_j \Big|^{-q} & = \frac{1}{\mb{E}|S|^{-q}} \mb{E}\Big| \frac{S}{(S^p+\mathcal{E})^{1/p}}\Big|^{-q} \mb{E} \Big| \sum_{j=1}^n Y_j v_j \Big|^{-q}  \\ & = c_1(p,q,n) \mb{E}\Big|\sum_{j=1}^n Y_j v_j \Big|^{-q} = c_2(p,q,n) \mb{E}\Big|\sum_{j=1}^n Y_j^{(p)} v_j \Big|^{-q},
\end{align*}
where $c_i(p,q,n)>0$ is independent of $v_1,\ldots,v_n$. As a result one gets
\[
	\vol_{n-k}(B_p^n \cap H) = c_3(k,p,n) \lim_{q\to k-} \frac{k-q}{k\vol_k(B_2^k)}  \mb{E}\Big|\sum_{j=1}^n Y_j^{(p)} v_j \Big|^{-q}.
\]
Taking $v_j= e_j$ for $1 \leq i \leq k$ and $v_j=0$ for $k+1 \leq j \leq n$ and using Lemma \ref{lm:f0-neg-mom} we obtain
\[
	\vol_{n-k}(B_p^{n-k}) = c_3(k,p,n) \lim_{q\to k-} \frac{k-q}{k\vol_k(B_2^k)}  \mb{E}\Big|(Y_1^{(p)}, \ldots, Y_k^{(p)}) \Big|^{-q} = c_3(k,p,n).
\]
\end{proof}

\begin{corollary}\label{cor:secBpn}
Let $p > 0$. For a unit vector $a \in \R^n$, we have
\[
\frac{\vol_{n-1}(B_p^n \cap a^\perp)}{\vol_{n-1}(B_p^{n-1})} = f_{a}(0),
\]
where $f_a$ is the density of $\sum_{j=1}^n a_jY_j^{(p)}$. 
\end{corollary}
\begin{proof}
This formula follows by combining Lemma \ref{lm:Bpn-sec-negative-moments} with Lemma \ref{lm:f0-neg-mom}. The correctness of the normalization constant can be checked by plugging in $a=e_1$.
\end{proof}

As an application, we show how to obtain the following theorem of Meyer and Pajor from \cite{MP88}. The main idea of exploiting Kanter's peakedness from \cite{Kant} comes from the original proof of Meyer and Pajor. In addition to illustrating our approach via negative moments, which we will build upon later, we hope this proof might be of independent interest.

\begin{theorem}[Meyer-Pajor, \cite{MP88}]\label{thm:Bpn-minimal-p>2}
Let $1 \leq k \leq n$ and let $H$ be a subspace in $\R^n$ of codimension~$k$. Then the following function
\[
	p \mapsto \vol_{n-k}(B_p^n \cap H)/\vol_{n-k}(B_{p}^{n-k})
\] 
is nondecreasing on $(0,\infty)$.
\end{theorem}   
\begin{proof}
For $\beta>\alpha$ the random variable $Y_j^{(\beta)}$ is more peaked than $Y_j^{(\alpha)}$ (see \cite{Kant} and \cite{MP88}).  Thus for every vectors $v_1, \dots, v_n$ in $\R^k$, $\sum_{j=1}^n Y_j^{(\beta)}v_j$ is more peaked than $\sum_{j=1}^n Y_j^{(\alpha)} v_j$. Consequently, for a norm $\|\cdot\|$ on $\R^k$ and $0 < q < k$,
\begin{equation}\label{eq:uni-gauss}
\E\left\|\sum_{j=1}^n Y_j^{(\beta)} v_j \right\|^{-q} \geq \E\left\|\sum_{j=1}^n Y_j^{(\alpha)} v_j \right\|^{-q}.
\end{equation}
Thus, the function $\alpha \mapsto \E\left\|\sum_{j=1}^n Y_j^{(\alpha)}v_j \right\|^{-q}$ is nondecreasing on $(0,\infty)$. 
Using this 
%with  $\alpha=2$ and $\beta=q$ 
together with Lemma \ref{lm:Bpn-sec-negative-moments},
% and the argument from the previous section, 
we get that
\begin{align*}
	p \mapsto \frac{\vol_n(B_p^n \cap H)}{\vol_{n-k}(B_q^{n-k})} & =  \lim_{q\to k-} \frac{k-q}{k\vol_k(B_2^k)}  \mb{E}\Big|\sum_{j=1}^n Y_j^{(p)} v_j \Big|^{-q} 
%  \geq \lim_{p\to k-} \frac{k-p}{k\vol_k(B_2^k)}  \mb{E}\Big|\sum_{j=1}^n Y_j^{(2)}v_j \Big|^{-p}
\end{align*}
is nondecreasing. 
%The random vector $\sum_{j=1}^n Y_j^{(2)} v_j$ is Gaussian in $\R^k$ with covariance matrix proportional to $\sum_{j=1}^n v_jv_j^\top = I$, thus its density equals $f(x)=e^{-\pi|x|^2}$, $x \in \R^k$ (the constant can be checked by plugging in $v_i=e_i$, $1,\ldots, k$ and $v_i=0$ for $k+1 \leq i \leq n$, in which case $\sum_{j=1}^n Y_j^{(2)}v_j=(Y_1^{(2)},\ldots, Y_n^{(2)})$). Since $f(0)=1$, the assertion follows from Lemma \ref{lm:f0-neg-mom}.
\end{proof}

\subsection{Sections of \texorpdfstring{$B_p^n$}{Bpn} via Gaussian mixtures} 
In the sequel we shall need one more formula in the special case of $B_p^n$ with $0<p<2$. This formula was mentioned in \cite{ENT18} (a hyperplane case) and \cite{NT} (a general case). We sketch a slightly different argument below, based again on negative moments, for simplicity for hyperplane sections.

We first need some notation. For $\alpha \in (0,1)$, let $g_\alpha$ be the density of a standard positive $\alpha$-stable random variable, that is a positive random variable $W_\alpha$ with the Laplace transform $\E e^{-uW_\alpha} = e^{-u^\alpha}$, $u > 0$. Let $V_1, \dots, V_n$ be i.i.d. positive random variables with density proportional to $t^{-3/2}g_{p/2}(t^{-1})$ and set $R_i=\sqrt{V_i/2}$. Take $G_i$ to be standard Gaussian random variables, independent of the $V_j$. According to Lemma 23(a) from \cite{ENT18}, the random variables $R_i G_i$ have densities $\beta_p^{-1}e^{-|x|^p}$. We also let
$
\bar V_j = (\E V_j^{-1/2})^2V_j
$
be normalised so that $\E \bar V_j^{-1/2} = 1$.

\begin{lemma}[Eskenazis-Nayar-Tkocz, \cite{ENT18}]\label{lm:sec-GM}
Let $0 < p < 2$. For a unit vector $a = (a_1, \dots, a_n)$ in $\R^n$, we have
\begin{equation}\label{eq:Bp-sec-p<2}
\frac{\vol_{n-1}(B_p^n \cap a^\perp)}{\vol_{n-1}(B_p^{n-1})} = \E\left(\sum_{j=1}^n a_j^2\bar V_j\right)^{-1/2}.
\end{equation}
\end{lemma}
\begin{proof}
Using Lemma \ref{lm:Bpn-sec-negative-moments} and the above Gaussian mixture representation for the $Y_j^{(p)}$,
\begin{align*}
\frac{\vol_{n-1}(B_p^n \cap a^\perp)}{\vol_{n-1}(B_p^{n-1})} &= \lim_{q\to 1-} \frac{1-q}{2}  \mb{E}\Big|\sum_{j=1}^n a_jY_j^{(p)} \Big|^{-q} \\
&= \kappa_p\lim_{q\to 1-} (1-q)  \mb{E}\Big|\sum_{j=1}^n a_j\sqrt{V_j}G_j \Big|^{-q}
\end{align*}
for a positive constant $\kappa_p$ which depends only on $p$ (resulting from rescalings of the random variables involved). Since $\sum_{j=1}^n a_j\sqrt{V_j}G_j $ has the same distribution as $\sqrt{\sum a_j^2V_j}G_1$ and $(1-q)\E|G_1|^{-q}$ converges to $\sqrt{\frac{2}{\pi}}$ (twice the density at $0$) as $q \to 1-$, after further rescalings, we obtain
\[
\frac{\vol_{n-1}(B_p^n \cap a^\perp)}{\vol_{n-1}(B_p^{n-1})} = \kappa_p'\E\left(\sum_{j=1}^n a_j^2\bar V_j\right)^{-1/2}.
\]
Plugging in $a = e_1$ shows that $\kappa_p' = 1$.
\end{proof}

\begin{remark}\label{rem:V-integrability}
The above expectation is finite due to the fact that $\E W_\alpha^r<\infty$ iff $r<\alpha$. Indeed,
\[
\int_{0}^\infty t^{q-3/2}g_{p/2}(t^{-1}) \dd t = \int_{0}^\infty t^{-q-1/2}g_{p/2}(t) \dd t = \E W_{p/2}^{-q-1/2}
\]
thus $\E V_1^q < \infty$ as long as $-q - 1/2 < p/2$, that is $q > -\frac{p+1}{2}$. The above fact can be deduced from the asymptotic formulas (see, e.g. \cite{Mi})
\[
	g_\alpha(t) \sim_{t \to \infty} M_\alpha t^{-(1+\alpha)}, \qquad g_\alpha(t) \sim_{t \to 0^+} K_\alpha t^{-\frac{2-\alpha}{2(1-\alpha)}} \exp(A_\alpha t^{-\frac{\alpha}{1-\alpha}}).
\]
\end{remark}

\section{Stability: heuristic explanation of the proof}\label{sec:heur}

We are ready to proceed with the proofs of Theorem \ref{thm:stab}. First, we briefly outline them. We emphasise that, as already highlighted in the introduction, as different and disconnected from each other our arguments may seem, their common probabilistic underpinning is the negative moment approach which yields very convenient formulae for sections, amenable to a detailed analysis allowing not only to find the extremisers, but also to develop precise first order error terms.

%Our proofs involve several different approaches building on various analytic and probabilistic formulae from Section \ref{sec:form} for the volume of sections. To a large extent, the probabilistic underpinning for these new geometric results is a connection to negative moments of weighted sums of independent random variables highlighted earlier. 

To give a short overview: 
\eqref{eq:st-Bp-p<2-max} simply follows from Schur convexity, its reversal, 
\eqref{eq:st-Bp-p<2-min} is obtained from a formula involving negative moments combined with complete monotonicity allowing to invoke the Laplace transform to leverage independence, 
\eqref{eq:st-Bp-p>2} for $2 < p < \infty$ relies on viewing the volume of sections as the $\infty$-norm of an appropriate probability density which is estimated using peakedness and additional probabilistic tools, e.g. the Berry-Esseen theorem, whereas \eqref{eq:st-Bp-p>2} for $p = \infty$ follows from a more general stability result for an underlying Khinchin-type inequality, obtained thanks to negative moments, and, finally, \eqref{eq:st-Q-max} is established by a careful analysis of Ball's proof, souped-up with new insights gained from representations via negative moments allowing for certain self-improvements of Ball's inequality (in the spirit of \cite{DDS} which establishes an analogous stability result for Szarek's $L_1-L_2$ classical Khinchin inequality, with arguments based on discrete Fourier analysis). We begin with the results for the cube.

\section{Cube slicing}\label{sec:Q}

\subsection{Minimal hyperplane cube sections}\label{sec:Qmin}
Prior to Vaaler's work \cite{V79}, Hadwiger in \cite{H72} and independently Hensley in \cite{H79} established that the minimal hyperplane sections of the cube are attained for coordinate subspaces. A different simple proof was later given in \cite{B86} (which was based on a direct minimisation of $\|f\|_\infty$ over even unimodal probability densities with fixed variance). Our method involving negative moments offers another simple approach with the advantage that it is well-suited to give a stability result. First we establish a robust version of a relevant Khinchin inequality.

\begin{theorem}\label{thm:spheres-stab}
Let $0 < p < 2$ and let $\xi_1, \dots, \xi_n$ be i.i.d. random vectors in $\R^d$ uniform on $S^{d-1}$, $d \geq 3$. For every $n \geq 1$ and real numbers $a_1, \dots, a_n$ such that $a_1^2+\dots + a_n^2 = 1$, we have
\[
\E\left|\sum_{j=1}^n a_j\xi_j\right|^{-p} \geq 1 + \frac{p(p+2)(2d-p-4)}{9d^2}\left(1-\sum_{j=1}^na_j^4\right).
\]
\end{theorem}
\begin{proof}
First we remark that a sharp inequality without the remainder term is a simple consequence of convexity. Indeed, for any $p>0$ we have
\begin{equation}\label{eq:hensley-easy}
\E\left|\sum_{j=1}^n a_j\xi_j\right|^{-p} = \E\left(\left|\sum_{j=1}^n a_j\xi_j\right|^2\right)^{-p/2} \geq \left(\E\left|\sum_{j=1}^n a_j\xi_j\right|^2\right)^{-p/2} = 1.
\end{equation}
To control the error in this estimate, a natural idea presents itself: we write 
\[
\left|\sum_{j=1}^n a_j\xi_j\right|^2 = 1 + Y
\]
with 
\[
Y = 2\sum_{i < j} a_ia_j\scal{\xi_i}{\xi_j}
\]
and seek a refinement of the pointwise bound $(1+x)^{-p/2} \geq 1 - \frac{p}{2}x$, $x > -1$ (resulting just from convexity) which gives \eqref{eq:hensley-easy}, in view of the fact that $Y > -1$ a.s. and $\E Y = 0$. We shall use the following lemma, the proof of which we defer for now (for simplicity, we did not try to optimise the numerical constants).

%\begin{lemma}
%For every $p > 0$ and $x > -1$, we have
%\[
%(1+x)^{-p/2} \geq 1- \frac{p}{2} x + \frac{p(p+2)}{8}  x^2 - \frac{p(p+2)(p+4)}{48} x^3.
%\]
%\end{lemma}

\begin{lemma}\label{lm:2point}
For every $p > 0$ and $x > -1$, we have
\[
(1+x)^{-p/2} \geq 1 - \frac{p}{2}x + \frac{p(p+2)}{9}x^2 -\frac{p(p+2)(p+4)}{72}x^3.
\]
\end{lemma}

This lemma yields
\[
\E\left|\sum_{j=1}^n a_j\xi_j\right|^{-p} = \E(1+Y)^{-p/2} \geq 1 + \frac{p(p+2)}{9}\E Y^2 - \frac{p(p+2)(p+4)}{72}\E Y^3.
\]
To compute $\E Y^2$ and $\E Y^3$, first note that thanks to rotational invariance and independence, for $i < j$,
\[
\E\scal{\xi_i}{\xi_j}^2 = \E \scal{\xi_i}{e_1}^2 = \frac{1}{d}
\]
and for $i < j < k$,
\begin{align*}
\E\scal{\xi_i}{\xi_j}\scal{\xi_j}{\xi_k}\scal{\xi_i}{\xi_k} &= \E\scal{\xi_i}{\xi_j}\scal{\xi_j}{e_1} \scal{\xi_i}{e_1} \\
&= \E\scal{\xi_j}{e_1}^2 \scal{\xi_i}{e_1}^2 + \sum_{l=2}^d \E\scal{\xi_i}{e_l}\scal{\xi_i}{e_1}\E\scal{\xi_j}{e_l}\scal{\xi_j}{e_1} \\
&= \E\scal{\xi_j}{e_1}^2 \E\scal{\xi_i}{e_1}^2 = \frac{1}{d^2},
\end{align*}
where in the second line we write $\scal{\xi_i}{\xi_j} = \sum_{l=1}^d \scal{\xi_i}{e_l}\scal{\xi_j}{e_l}$, use independence and the fact that vectors $\xi_i$ have uncorrelated components to see that the sum over $l \geq 2$ vanishes.
Thus, using symmetry again,
\[
\E Y^2 = 4\sum_{i<j}a_i^2a_j^2 \E\scal{\xi_i}{\xi_j}^2 = \frac{4}{d}\sum_{i<j}a_i^2a_j^2
\]
and
\[
\E Y^3 = 8\cdot 6\sum_{i<j<k}a_i^2a_j^2a_k^2\E\scal{\xi_i}{\xi_j}\scal{\xi_j}{\xi_k}\scal{\xi_i}{\xi_k} = \frac{48}{d^2}\sum_{i<j<k}a_i^2a_j^2a_k^2.
\]
Introducing, $s_l = \sum_{i=1}^n a_i^{2l}$, $l = 1, 2, \dots$, we have $s_1 = 1$ and using Newton identities for symmetric functions, we express $2\sum_{i < j} a_i^2a_j^2 = 1 - s_2$, $6\sum_{i<j<k}a_i^2a_j^2a_k^2 = 1-3s_2+2s_3$. Moreover, $s_3 \leq s_2$. As a result,
\begin{align*}
\E Y^2 &= \frac{2}{d}(1-s_2), \\
\E Y^3 &= \frac{8}{d^2}(1-3s_2+2s_3) \leq \frac{8}{d^2}(1-s_2).
\end{align*}
Therefore,
\begin{align*}
\E\left|\sum_{j=1}^n a_j\xi_j\right|^{-p} &\geq 1 + \frac{2p(p+2)}{9d}(1-s_2) - \frac{p(p+2)(p+4)}{9d^2}(1-s_2) \\
&= 1+\frac{p(p+2)(2d-p-4)}{9d^2}(1-s_2).
\end{align*}
\end{proof}

Now we are able to deduce a stability result for minimal hyperplane sections of the cube, \eqref{eq:st-Bp-p>2} for $p=\infty$. For convenience, we restate this here.
\begin{theorem}\label{thm:hensley-stab}
Let $a = (a_1,\dots,a_n)$ be a unit vector in $\R^n$ with $a_1 \geq a_2 \geq \dots \geq 0$. Then,
\[
\vol_{n-1}\left(Q_n \cap a^\perp\right) \geq 1 + \frac{1}{54}|a-e_1|^2.
\]
\end{theorem}
\begin{proof}
Note that under the assumption on $a$,
\begin{align*}
\frac{1}{2}|a-e_1|^2 = \frac12\left((1-a_1)^2 + \sum_{i =2}^n a_i^2\right) = 1-a_1\leq 1-a_1^2 &= 1 - \sum_i a_1^2a_i^2 \leq 1 - \sum_i a_i^4.
\end{align*}
Thus the assertion follows immediately from Theorem \ref{thm:spheres-stab} applied to $p=1$ and $d=3$, in view of Lemma \ref{lm:sec-cube-codim1}.  
\end{proof}

\begin{remark}\label{rem:cube-opt}
The dependence on $\delta(a) = 1 - \sum_{j=1}^n a_j^4$ in Theorem \ref{thm:spheres-stab} modulo a constant factor is best possible: there are examples of unit vectors $a$ with $\delta(a) \to 0$ for which $\E|\sum a_j\xi_j|^{-p} - 1 = O_{p,d}(\delta(a))$. For instance, take $a = (\sqrt{1-\e}, \sqrt{\e}, 0, \dots, 0)$ with $\e<\frac{1}{16}$. Since for $0<p<2$ and $x \in [-\frac12,1]$ one has $(1+x)^{-\frac{p}{2}} \leq 1-\frac{p}{2}x+8x^2$ (use Taylor formula with Lagrange remainder), it follows that
\[
\E|\sum a_j\xi_j|^{-p} = \E(1 + 2\sqrt{\e(1-\e)}\scal{\xi_1}{\xi_2})^{-p/2} \leq 1+32 \e(1-\e)\E\scal{\xi_1}{\xi_2}^2 = 1+\frac{32 \e(1-\e)}{d}. 
\]
Since $1-\sum_{j=1}^n a_j^4 = 2\e(1-\e)$, we get $\E|\sum a_j\xi_j|^{-p} \leq 1+ \frac{16}{d}(1-\sum_{j=1}^n a_j^4 )$.

In particular, the same remark applies to Theorem \ref{thm:hensley-stab} as well.
\end{remark}

%The dependence on $\delta(a) = 1 - \sum_{j=1}^n a_j^4$ in Theorem \ref{thm:spheres-stab} is optimal (modulo a constant factor) in the following sense: there is a positive constant $C_{p,d}$ which depends only on $p$ and $d$ such that
%\[
%\E\left|\sum_{j=1}^n a_j\xi_j\right|^{-p} \leq 1 + C_{p,d}\delta(a).
%\]
%We obtain this by following the proof of Theorem \ref{thm:spheres-stab} and using, instead of Lemma \ref{lm:2point}, the following elementary inequality,
%\[
%(1+x)^{-p/2} \leq 1 - \frac{p}{2}x + \frac{p(p+2)}{8}x^2, \qquad p > 0, x > -1.
%\]

It remains to prove the point-wise inequality we used.

\begin{proof}[Proof of Lemma \ref{lm:2point}]
From the Taylor formula with Lagrange reminder for the function $(1+x)^{-\frac{p}{2}}$ one gets that for $x \leq \frac{2}{p+4}$ 
\[
(1+x)^{-p/2} - 1 + \frac{p}{2} x \geq  \frac{p(p+2)}{8}  x^2 - \frac{p(p+2)(p+4)}{48} x^3 \geq \frac{p(p+2)}{9}x^2 -\frac{p(p+2)(p+4)}{72}x^3.
\]
We now show how to treat the case $x \geq 0$. Define 
\[
\psi(x)= (1+x)^{-p/2} - 1 + \frac{p}{2}x - \frac{p(p+2)}{9}x^2 +\frac{p(p+2)(p+4)}{72}x^3.
\] 
Our goal is to prove that $\psi(x) \geq 0$ for $x \geq 0$. Note that $\psi(0)=\psi'(0)=0$. Thus it suffices to show that for $x \geq 0$ we have $\psi''(x) \geq 0$. This is equivalent to $(1+x)^{-\frac{p+4}{2}} \geq \frac89-\frac13(p+4)x$. Define $\alpha=\frac12(p+4)$. Our inequality reads $(1+x)^{-\alpha} \geq \frac89-\frac23 \alpha x$. We shall verify this for arbitrary $\alpha, x>0$. Let $t=\alpha x$. Rewriting gives $(1+\frac{t}{\alpha})^{-\alpha} \geq \frac89 - \frac23 t$. We have $(1+\frac{t}{\alpha})^{-\alpha} \geq e^{-t}$ (take the logarithm and use the inequality $\ln(1+y)\leq y$) and thus it is enough to show that $e^{-t} \geq \frac89 - \frac23 t$ for $t >0$. The function $h(t)=e^{-t} - \frac89 + \frac23 t$ has a minimum for $t=\ln(\frac32)$. It is enough to verify that $\frac23 \geq \frac89 - \frac23 \ln(\frac32)$.  This is $\ln(\frac32) \geq \frac13$ which is true.
\end{proof}

\subsection{Maximal hyperplane cube sections}\label{sec:Qmax}

Our goal here is to prove \eqref{eq:st-Q-max}. We recall two formulae (see Lemma \ref{lm:sec-cube-codim1} and Remark \ref{rem:Qn-sec}),
\begin{align}
\vol_{n-1}(Q_n \cap a^\perp) &= \E\left|\sum_{j=1}^n a_j\xi_j\right|^{-1} \label{eq:Q-spheres}\\
&= \frac{2}{\pi}\int_0^\infty \prod_{j=1}^n \frac{\sin(a_jt)}{a_jt} \dd t, \label{eq:Q-Fourier}
\end{align}
as well as the fact that
\begin{equation}\label{eq:bus}
\|a\|_{\text{Bus}} = \frac{|a|}{\vol_{n-1}(Q_n \cap a^\perp)}
\end{equation}
defines a norm on $\R^n$, thanks to Busemann's theorem (see \cite{Bus}, or, e.g. Theorem 3.9 in \cite{MilPa}). It follows that the function $a \mapsto \vol_{n-1}(Q_n \cap a^\perp)$ is $2$-Lipschitz on the unit sphere.

\begin{lemma}\label{lm:Lip}
For every unit vectors $a$, $b$ in $\R^n$, we have
\[
\left| \vol_{n-1}(Q_n \cap a^\perp) - \vol_{n-1}(Q_n \cap b^\perp) \right| \leq 2|a-b|.
\]
\end{lemma}
\begin{proof}
Letting $F(a) = \vol_{n-1}(Q_n \cap a^\perp)$, by the triangle inequality we have
\[
\frac{\left| F(a) - F(b) \right|}{F(a)F(b)} =\left|\|a\|_{\text{Bus}} - \|b\|_{\text{Bus}}\right| \leq \|a-b\|_{\text{Bus}} = \frac{|a-b|}{F(a-b)}.
\]
Using that $1 \leq F(x) \leq \sqrt{2}$ for every vector $x$ concludes the proof.
\end{proof}

We will also need the following observation.

\begin{lemma}\label{lm:X+Y}
Let $X$ and $Y$ be two independent rotationally invariant random vectors in $\R^3$. Then
\[
\E|X+Y|^{-1} = \E\min\left\{|X|^{-1},|Y|^{-1}\right\} \leq \min\{ \E |X|^{-1}, \E |Y|^{-1}\}.
\]
In particular,
\[
\vol_{n-1}(Q_n \cap a^\perp) \leq \min\{|a_j|^{-1}\}.
\]
\end{lemma}

\begin{proof}
Since $X$ and $Y$ are rotationally invariant, their distributions can be written as $|X| \xi_1$ and $|Y| \xi_2$, where $\xi_1, \xi_2$ are uniform on $S^2$,  chosen independently of $X$ and $Y$. By conditioning  on $X$ and $Y$, it suffices to verify the identity $\E_{\xi_1, \xi_2} |r \xi_1+s \xi_2|^{-1}=\min(r,s)^{-1}$. Note that by rotation invariance $\scal{\xi_1}{\xi_2}$ has the same distribution as $\scal{\xi_1}{e_1}$, that is a uniform distribution on $[-1,1]$. Therefore
\begin{align*}
	\E_{\xi_1, \xi_2} |r \xi_1+s \xi_2|^{-1} & = \E_{\xi_1, \xi_2} (|r \xi_1+s \xi_2|^2)^{-1/2} = \frac12\int_{-1}^1 (r^2+s^2+2rs u)^{-1/2} \dd u  \\
	& = \frac{(r^2+s^2+2rs u)^{1/2}}{2rs}\Big|_{-1}^1 = \frac{|r+s|-|r-s|}{2rs} = \frac{\min\{r,s\}}{rs} = \min\{r^{-1},s^{-1}\}^{-1}.
\end{align*}

To prove the second part it suffices to take $X=\sum_{j=1}^{n-1} a_j \xi_j$, $Y= a_n \xi_n$ and use the inequality $\E|X+Y|^{-1} \leq \E|Y|^{-1}$.
\end{proof}

Since the maximal section has volume $\sqrt{2}$, that is $\vol_{n-1}(Q_n \cap (\frac{e_1+e_2}{\sqrt{2}})^\perp) = \sqrt{2}$, our stability result \eqref{eq:st-Q-max} for maximal sections of the cube can be equivalently stated as follows
\begin{equation}\label{eq:st-Q-max'}
\vol_{n-1}(Q_n \cap a^\perp) \leq \sqrt{2} - c_0\left|a - \frac{e_1+e_2}{\sqrt{2}}\right|,
\end{equation}
for every $n$ and every unit vector $a$ in $\R^n$ with $a_1 \geq a_2 \geq \dots \geq a_n \geq 0$, for some universal constant $c_0$.

The proof involves different arguments, depending on whether $a$ is \emph{close} to the extremiser or not and whether its largest coordinate is \emph{large} or not. We assume throughout that $a$ is a unit vector in $\R^n$ with $a_1 \geq a_2 \geq \dots \geq a_n \geq 0$ and set
\[
\delta(a) = \left|a - \frac{e_1+e_2}{\sqrt{2}}\right|^2 = 2 - \sqrt{2}(a_1+a_2).
\]

For vectors $a$ \emph{close} to the extremiser, we have the following \emph{local} stability result (it is to some extent in the spirit of Lemma 3.7 from \cite{DDS}).

\begin{lemma}\label{lm:delta<}
There are universal constants $\delta_0 \in (0,\frac{1}{\sqrt{2}})$ and $c_0>0$ such that \eqref{eq:st-Q-max'} holds for every $a$ with $\delta(a) \leq \delta_0$.
\end{lemma}

For vectors $a$ \emph{away} from the extremiser with largest coordinate sufficiently close to $\frac{1}{\sqrt{2}}$, we prove the following lemma.

\begin{lemma}\label{lm:delta>}
Let $\delta_0$ be the constant from Lemma \ref{lm:delta<}. There are positive universal constants $\gamma_0, c_1$  such that 
\begin{equation}\label{eq:delta>}
\vol_{n-1}(Q_n \cap a^\perp) \leq \sqrt{2} - c_1
\end{equation}
holds for every $a$ with $\delta(a) > \delta_0$ and $a_1 \leq \frac{1}{\sqrt{2}} + \gamma_0$.
\end{lemma}

The remaining case is straightforward: taking these two lemmas for granted, it is very easy to prove \eqref{eq:st-Q-max'}.

\begin{proof}[Proof of \eqref{eq:st-Q-max'}]
In view of Lemmas \ref{lm:delta<} and \ref{lm:delta>}, it remains to consider the case when $a_1 > \frac{1}{\sqrt{2}} + \gamma_0$. From Lemma \ref{lm:X+Y}, we have
\[
\vol_{n-1}(Q_n \cap a^\perp) \leq \frac{1}{a_1} < \frac{1}{1/\sqrt{2} + \gamma_0} < \sqrt{2} - \gamma_0 < \sqrt{2} - \frac{\gamma_0}{\sqrt{2}}\sqrt{\delta(a)},
\]
because $\delta(a) < 2$, so in this case \eqref{eq:st-Q-max'} also holds.
\end{proof}

It remains to prove the lemmas.

\begin{proof}[Proof of Lemma \ref{lm:delta<}]
The idea is to argue that Ball's inequality $\vol_{n-1}(Q_n \cap a^\perp) \leq \sqrt{2}$ allows for a self-improvement near the extremiser. We shall assume that $n \geq 3$ and $a_1^2 + a_2^2 < 1$ (the case $n = 2$ can be analysed directly). A starting point is formula \eqref{eq:Q-spheres}, combined with Lemma \ref{lm:X+Y},
\[
\vol_{n-1}(Q_n \cap a^\perp) = \E_{X,Y}\min\left\{|X|^{-1},|Y|^{-1}\right\},
\]
where we apply it to $X = a_1\xi_1 + a_2\xi_2$ and $Y = \sum_{j=3}^n a_j\xi_j$. By Ball's inequality, 
\[
\E_Y|Y|^{-1} \leq \sqrt{2}(1-a_1^2-a_2^2)^{-1/2}.
\]
Thus, thanks to the independence of $X$ and $Y$ and the simple inequality 
\[
\E_Y\min\left\{|X|^{-1},|Y|^{-1}\right\} \leq \min\left\{|X|^{-1},\E_Y|Y|^{-1}\right\},
\]
we obtain
\[
\vol_{n-1}(Q_n \cap a^\perp) \leq \E_X \min\left\{|X|^{-1},\sqrt{2}(1-a_1^2-a_2^2)^{-1/2}\right\}.
\]
Note that $|X|$ has the same distribution as $(a_1^2+a_2^2+2a_1a_2U)^{1/2}$, where $U$ is a random variable uniform on $[-1,1]$. To evaluate $\E_X$, observe that $|X|^{-1} < \sqrt{2}(1-a_1^2-a_2^2)^{-1/2}$ corresponds to $U > u_0$, where
\[
u_0 = \frac{1-3(a_1^2+a_2^2)}{4a_1a_2}.
\]
We need to consider two cases. Let $\delta = \delta(a)/2$, that is 
\[
a_1 + a_2 = \sqrt{2}(1-\delta).
\]

\emph{Case 1: $u_0 \leq -1$.} Then 
\[
\E_X \min\left\{|X|^{-1},\sqrt{2}(1-a_1^2-a_2^2)^{-1/2}\right\} = \E|X|^{-1} = \min(a_1,a_2)^{-1}= a_1^{-1}.
\]
Given $a_1 + a_2 = \sqrt{2}(1-\delta)$,
the condition $u_0 \leq -1$ implies that $a_1 \geq \bar a_1$, where $\bar a_1$ is the larger of the two solutions to the quadratic equation
\[
1-3(a_1^2+(\sqrt{2}(1-\delta)-a_1)^2) = -4a_1(\sqrt{2}(1-\delta)-a_1).
\]
This yields
\[
\vol_{n-1}(Q_n \cap a^\perp)  \leq \frac{1}{\bar a_1} = \sqrt{2}\left(1-\delta + \sqrt{\frac{\delta}{5}}\sqrt{2-\delta}\right)^{-1} \leq \sqrt{2} - c_0\sqrt{\delta}
\]
for a universal constant $c_0 > 0$, provided that $\delta$ is sufficiently small.

\emph{Case 2: $u_0 > -1$.} It is clear that for all $\delta$ sufficiently small, $u_0 < 1$ (in fact since $a_1+a_2 \leq \sqrt{2}(a_1^2+a_2^2) \leq \sqrt{2}$, the equality  $a_1+a_2=\sqrt{2}(1-\delta)$ for small $\delta$ implies that both numbers $a_1, a_2$ are close to $\frac{1}{\sqrt{2}}$ and thus $u_0$ is close to $-1$). Then
\begin{align*}
\E_X &\min\left\{|X|^{-1},\sqrt{2}(1-a_1^2-a_2^2)^{-1/2}\right\} \\
&= \frac{1}{2}(u_0+1)\sqrt{2}(1-a_1^2-a_2^2)^{-1/2} + \frac{1}{2}\int_{u_0}^1 (a_1^2+a_2^2 + 2a_1a_2u)^{-1/2} \dd u \\
&= \frac{u_0+1}{\sqrt{2(1-a_1^2-a_2^2)}} + \frac{a_1+a_2-\sqrt{a_1^2+a_2^2+2a_1a_2u_0}}{2a_1a_2}.
\end{align*}
Plugging in $u_0$ and rewriting in terms of $s = a_1 + a_2$, $\rho = a_1^2+a_2^2$ results with an upper bound on $\vol_{n-1}(Q_n \cap a^\perp)$ by
\[
h(s,\rho) = \frac{s}{s^2-\rho} + \frac{2s^2-1-3\rho}{2\sqrt{2}(s^2-\rho)\sqrt{1-\rho}}.
\]
Note that $\frac{s^2}{2} \leq \rho < 1$. We claim that for every $1 \leq s \leq \sqrt{2}$, function $\rho \mapsto h(s,\rho)$ is decreasing on $(\frac{s^2}{2},1)$. Thus,
\begin{align*}
\vol_{n-1}(Q_n \cap a^\perp) \leq h(s,s^2/2) &= \frac{2}{s} - \frac{\sqrt{1-s^2/2}}{\sqrt{2}s^2} \\
&= \sqrt{2}(1-\delta)^{-2}\left(1-\delta-\frac{\sqrt{\delta}}{2 \sqrt{2}}\sqrt{2-\delta}\right) \\
&< \sqrt{2} - c_0\sqrt{\delta}
\end{align*}
for a universal constant $c_0 > 0$ and all sufficiently small $\delta$.

To prove that $\rho \mapsto h(s,\rho)$ is decreasing on $(\frac{s^2}{2},1)$, we fix $1 \leq s \leq \sqrt{2}$ and compute the derivative
\[
\frac{\partial h}{\partial \rho} = 
-\frac{2\sqrt{2}-3\sqrt{2}\rho(1+\rho)-8(1-\rho)^{3/2}s+3\sqrt{2}s^2(1+\rho)-2\sqrt{2}s^4}{8(1-\rho)^{3/2}(s^2-\rho)^2}.
\]
Note that the numerator 
\[
\tilde{h}(s,\rho)=2\sqrt{2}-3\sqrt{2}\rho(1+\rho)-8(1-\rho)^{3/2}s+3\sqrt{2}s^2(1+\rho)-2\sqrt{2}s^4
\]
is a concave function of $\rho\in(s^2/2,1)$, as a sum of concave functions. It suffices to show that the values at the endpoints are non-negative. At $\rho=1$, we have
\[
\tilde{h}(s,1)=-2\sqrt{2}(s^4-3s^2+2)=2\sqrt{2}(s^2-1)(2-s^2)\geq 0.
\]
At $\rho=s^2/2$, we get 
\[
\tilde{h}(s,s^2/2)=\frac{2-s^2}{2\sqrt{2}}\left(5s^2+4-8s\sqrt{2-s^2}\right) \geq \frac{2-s^2}{2\sqrt{2}}\left(5s^2-4\right) \geq 0,
\]
by $2s\sqrt{2-s^2} \leq s^2 + (2-s^2) = 2$.
\end{proof}

\begin{proof}[Proof of Lemma \ref{lm:delta>}]
Assume that $\delta(a) > \delta_0$. In particular,
\begin{equation}\label{eq:a2-small}
a_2 \leq \frac{1}{2}(a_1 + a_2) = \frac{2-\delta(a)}{2\sqrt{2}} < \frac{1}{\sqrt{2}} - \frac{\delta_0}{2\sqrt{2}}.
\end{equation}
The argument is now split into two cases: when $a_1 \leq \frac{1}{\sqrt{2}}$, we employ \eqref{eq:Q-Fourier} and use Ball's approach to show that savings simply come from $a_2$ being small, whilst when $a_1 > \frac{1}{\sqrt{2}}$, provided $a_1$ is \emph{close} to $\frac{1}{\sqrt{2}}$, we employ Busemann's theorem to reduce this case to the previous one.

\emph{Case 1: $a_1 \leq \frac{1}{\sqrt{2}}$.} For $s \geq 2$, we define
\[
\Psi(s) = \frac{2}{\pi}\sqrt{s}\int_0^\infty \left|\frac{\sin t}{t}\right|^s \dd t.
\]
To establish his cube-slicing result, Ball showed in \cite{B86} that
\[
\Psi(s) < \Psi(2) = \sqrt{2}, \qquad s > 2.
\]
Moreover, since $\frac{\sin(t\sqrt{s})}{t/\sqrt{s}} = 1 - \frac{t^2}{6s} + O(s^{-2})$ as $s \to \infty$,
\[
\lim_{s \to \infty} \Psi(s) = \sqrt{\frac{6}{\pi}} < \sqrt{2}.
\]
In particular, by continuity, for every $s_0 > 2$, there is $0 < \theta_0 < 1$ such that
\begin{equation}\label{eq:Psi}
\Psi(s) \leq \theta_0\sqrt{2}, \qquad s \geq s_0.
\end{equation}
As in \cite{B86}, applying H\"older's inequality in \eqref{eq:Q-Fourier} yields
\[
\vol_{n-1}(Q_n \cap a^\perp) \leq \prod_{j=1}^n \Psi(a_j^{-2})^{a_j^2}.
\]
Letting $s_0 = 2\left(1 - \delta_0/2\right)^{-2}$, from \eqref{eq:a2-small}, we know that $a_j^{-2} \geq s_0$ for each $j \geq 2$, thus \eqref{eq:Psi} applied to each $j \geq 2$ and $\Psi(a_1^{-2}) \leq \sqrt{2}$ give
\[
\vol_{n-1}(Q_n \cap a^\perp) \leq \theta_0^{1-a_1^2}\sqrt{2} \leq \theta_0^{1/2}\sqrt{2} = \sqrt{2} - c_1.
\]

\emph{Case 2: $\frac{1}{\sqrt{2}} < a_1$.} We argue that there are positive universal constants $\gamma_0, c_2$ such that if additionally $a_1 < \frac{1}{\sqrt{2}} + \gamma_0$, then $\vol_{n-1}(Q_n \cap a^\perp) \leq \sqrt{2} - c_2$. To this end, we modify $a$ and consider the unit vector 
\[
b = \left(\frac{1}{\sqrt{2}}, \sqrt{a_1^2+a_2^2 - \frac{1}{2}}, a_3, \dots, a_n\right).
\]
Note that $b_1 \geq b_2$ and since $b_2 \geq a_2$, also $b_2 \geq b_3 \geq \dots \geq b_n$. Moreover, crudely,
\[
\sqrt{a_1^2+a_2^2 - \frac{1}{2}} - a_2 = \frac{a_1^2-\frac{1}{2}}{\sqrt{a_1^2+a_2^2 - \frac{1}{2}} + a_2} \leq \sqrt{a_1^2-\frac{1}{2}} \leq \sqrt{2\gamma_0},
\]
thus
\[
|a-b|^2 = \left(a_1 - \frac{1}{\sqrt{2}}\right)^2 + \left(\sqrt{a_1^2+a_2^2 - \frac{1}{2}} - a_2\right)^2 < \gamma_0^2 + 2\gamma_0.
\]
Lemma \ref{lm:Lip} yields
\[
\vol_{n-1}(Q_n \cap a^\perp) \leq \vol_{n-1}(Q_n \cap b^\perp) + 2\sqrt{\gamma_0^2+2\gamma_0}.
\]
If $\delta(b) > \delta_0$, then Case 1 applied to $b$ gives
\[
\vol_{n-1}(Q_n \cap b^\perp) < \sqrt{2} - c_1.
\]
Otherwise, observing that
\begin{align*}
\delta(b) &= \delta(a) - \sqrt{2}\left(\frac{1}{\sqrt{2}}+\sqrt{a_1^2+a_2^2 - \frac{1}{2}}-a_1-a_2\right)\\
&> \delta_0 -\sqrt{2}\left(\sqrt{a_1^2+a_2^2 - \frac{1}{2}}-a_2\right) \\
&> \delta_0 - 2\sqrt{\gamma_0},
\end{align*}
Lemma \ref{lm:delta<} applied to $b$ gives
\[
\vol_{n-1}(Q_n \cap b^\perp) < \sqrt{2} - c_0\sqrt{\delta_0 - 2\sqrt{\gamma_0}}.
\]
In any case, choosing $\gamma_0$ sufficiently small (depending on the values of $c_0, c_1, \delta_0$), we can ensure that
\[
\vol_{n-1}(Q_n \cap a^\perp) \leq \sqrt{2} - c_2
\]
with a positive universal constant $c_2$.
\end{proof}

\begin{remark}\label{rem:st-Q-max-opt}
The dependence on $\delta(a)$ in \eqref{eq:st-Q-max'} (modulo the universal constant $c_0$) is best possible: if we consider $a_\e = \left(\sqrt{\frac12+\e},\sqrt{\frac12-\e},0,\dots,0\right)$ with $\e \to 0$, then $\delta(a) = \e^2 + O(\e^4)$ and $\vol_{n-1}(Q_n \cap a^\perp) = a_1^{-1} = \sqrt{2}-\sqrt{2\delta(a)}+o(\sqrt{\delta(a)})$.
\end{remark}

\section{Hyperplane sections of  \texorpdfstring{$B_p^n$, $0 < p < \infty$}{Bpn, 0<p<oo}}\label{sec:0pinf}

\subsection{Case \texorpdfstring{$0 < p < 2$}{0<p<2}}

As remarked in \cite{ENT18}, formula \eqref{eq:Bp-sec-p<2} immediately yields the Schur-convexity of the function 
\[
(b_1,\dots,b_n) \mapsto \vol_{n-1}(B_p^n \cap (\sqrt{b_1},\dots,\sqrt{b_n})^\perp)
\]
on $\R_+^n$, in particular asserting that the subspaces of minimal and maximal volume cross-section are $(\frac{1}{\sqrt{n}},\dots,\frac{1}{\sqrt{n}})^\perp$ and $(1,0,\dots,0)$. Moreover, the formula allows to obtain stability results for these extremisers, which has not been observed before.

\subsubsection{Case \texorpdfstring{$0 < p < 2$}{0<p<2}: maximal sections}\label{sec:p<2max}

Thanks to Schur-convexity the case of maximal sections is straightforward.

%\begin{theorem}\label{thm:st-Bp-p<2-max}
%Let $0 < p < 2$. For every $n \geq 1$ and every unit vector $a$ in $\R^n$ such that $a_1 \geq a_2 \geq \dots \geq a_n \geq 0$, we have
%\begin{equation}\label{eq:st-Bp-p<2-max}
%\frac{\vol_{n-1}(B_p^n \cap a^\perp)}{\vol_{n-1}(B_p^{n-1})} \leq \left(a_1^p+(1-a_1^2)^{p/2}\right)^{-1/p}.
%\end{equation}
%\end{theorem}

\begin{proof}[Proof of \eqref{eq:st-Bp-p<2-max}]
By \eqref{eq:Bp-sec-p<2} and Schur-convexity,
\begin{align*}
\frac{\vol_{n-1}(B_p^n \cap a^\perp)}{\vol_{n-1}(B_p^{n-1})} &= \E\left(\sum_{j=1}^n a_j^2\bar V_j\right)^{-1/2} \leq \E\left( a_1^2\bar V_1 + (1-a_1^2)\bar V_2\right)^{-1/2} \\
&= \frac{\vol_{1}(B_p^2 \cap (a_1,\sqrt{1-a_1^2})^\perp)}{\vol_{1}(B_p^{1})},
\end{align*}
which is exactly the right hand side of \eqref{eq:st-Bp-p<2-max}. 
\end{proof}

\begin{remark}\label{rem:opt-p<2-max}
The bound is clearly optimal as it is attained in the case of vectors with at most two nonzero coordinates. Moreover, the right hand side of \eqref{eq:st-Bp-p<2-max} in terms of $\delta = \delta(a) = |a - e_1|$ is asymptotic to $1 - \frac{1}{p}\delta^p$ as $\delta \to 0^+$.
\end{remark}

\subsubsection{Case \texorpdfstring{$0 < p < 2$}{0<p<2}: minimal sections}\label{sec:p<2min}

Here our goal is to establish \eqref{eq:st-Bp-p<2-min}.
%\begin{theorem}\label{thm:st-Bp-p<2-min}
%Let $0 < p < 2$. There is a positive constant $c_p$ which depends only on $p$ such that for every $n \geq 1$ and every unit vector $a$ in $\R^n$, we have
%\begin{equation}\label{eq:st-Bp-p<2-min}
%\frac{\vol_{n-1}(B_p^n \cap a^\perp)}{\vol_{n-1}(B_p^n \cap (\frac{1}{\sqrt{n}},\dots,\frac{1}{\sqrt{n}})^\perp)} \geq 1 + c_p\sum_{j=1}^n (a_j^2-1/n)^2.
%\end{equation}
%\end{theorem}
We begin with a relevant stability result for negative moments. We rely on the fact that $x \mapsto x^{-q}$, $q>0$ is completely monotone, which allows to use simple convexity properties of log-moment generating functions.

\begin{lemma}\label{lm:st-exp}
Let $Y$ be a nonnegative random variable and $\Lambda(u) = \log\E e^{-u Y}$, $u \geq 0$. For every nonnegative real numbers $b_1, \dots, b_n$ with $B = \sum_{j=1}^n b_j$, we have
\begin{equation}\label{eq:st-exp-min}
\sum_{j=1}^n \Lambda(b_j) \geq n\Lambda(B/n) + c\sum_{j=1}^n (b_j - B/n)^2,
\end{equation}
%and
%\begin{equation}\label{eq:st-exp-max}
%\sum_{j=1}^n \Lambda(b_j) \leq \Lambda(B) - c\left(B^2-\sum_{j=1}^n b_j^2\right),
%\end{equation}
where
\[
c = \frac{1}{4}\sup_{0 < \alpha < \beta < \gamma} e^{-L(\alpha+\gamma)}(\beta-\alpha)^2\p{Y < \alpha}\p{\beta < Y < \gamma}
\]
with $L = \max_{j \leq n} b_j$.
% in \eqref{eq:st-exp-min} and $L = B$ in \eqref{eq:st-exp-max}.
\end{lemma}
\begin{proof}
By Taylor's theorem with Lagrange's reminder,
\[
\Lambda(b_j) = \Lambda(B/n) + (b_j-B/n)\Lambda'(B/n) + \frac{1}{2}(b_j-B/n)^2\Lambda''(\theta_j),
\]
for some $\theta_j$ between $b_j$ and $B/n$. Adding these inequalities over $j \leq n$ gives \eqref{eq:st-exp-min} with $c = \frac{1}{2}\inf_{(0,\max_j b_j)}\Lambda''$. Let $Y_1, Y_2$ be independent copies of $Y$. Crudely, $\E e^{-uY_1} \leq 1$, so for $0 < \alpha < \beta < \gamma$,
\begin{align*}
\Lambda''(u) &= \frac{1}{2}\frac{1}{(\E e^{-uY_1})^2}\E(Y_2-Y_1)^2e^{-uY_1}e^{-uY_2} \\
&\geq \frac{1}{2}\E(Y_2-Y_1)^2e^{-uY_1}e^{-uY_2}\1_{\{Y_1 < \alpha\}}\1_{\{\beta < Y_2 < \gamma\}} \\
&\geq \frac{1}{2}(\beta - \alpha)^2e^{-u(\alpha+\gamma)}\p{Y_1 < \alpha}\p{\beta < Y_2 < \gamma},
\end{align*}
which proves \eqref{eq:st-exp-min}. 
%To prove \eqref{eq:st-exp-max}, we argue similarly. First, for $\theta \in [0,1]$, we write
%\[
%\Lambda(B) = \Lambda(\theta B) + (B-\theta B)\Lambda'(\theta B) + \frac{1}{2}(B-\theta B)^2\Lambda''(\xi_+)
%\]
%and
%\[
%\Lambda(0) = \Lambda(\theta B) + (0-\theta B)\Lambda'(\theta B) + \frac{1}{2}(0-\theta B)^2\Lambda''(\xi_-)
%\]
%for some $\xi_\pm$ between $\theta B$ and $0$, $B$, respectively. Now, since $\Lambda(0) = 0$, multiplying the first inequality by $\theta$, the second one by $1-\theta$ and adding together yields
%\[
%\theta\Lambda(B) \geq \Lambda(\theta B) + \frac{1}{2}\theta(1-\theta)B^2\inf_{(0,B)} \Lambda''.
%\]
%Applying this to $\theta = b_j/B$, adding such inequalities  and lower-bounding $\Lambda''$ as before gives \eqref{eq:st-exp-max}.
\end{proof}

\begin{theorem}\label{thm:st-Bp-mom-p<2}
Let $q > 0$. Let $Y$ be a nonnegative random variable which is not constant a.s. with $\E Y < \infty$. Let $Y_1, Y_2, \dots$ be its i.i.d. copies. For every $b_1, \dots, b_n \geq 0$ with $\sum_{j=1}^n b_j  = 1$, we have
\begin{equation}\label{eq:st-mom-min}
\E\left(\sum_{j=1}^n b_jY_j\right)^{-q} \geq \E\left(\sum_{j=1}^n \frac{1}{n}Y_j\right)^{-q} + c_{q,Y}\sum_{j=1}^n(b_j-1/n)^2,
\end{equation}
%and
%\begin{equation}\label{eq:st-Bp-max}
%\E\left(\sum_{j=1}^n b_jY_j\right)^{-q} \leq \E\left(Y_1\right)^{-q} + c_{q,Y}\left(1-\sum_{j=1}^n b_j^2\right),
%\end{equation}
for some positive constant $c_{q,Y}$ which depends only on $q$ and the distribution of $Y$.
\end{theorem}
\begin{proof}
Using $x^{-q} = \Gamma(q)^{-1}\int_0^\infty e^{-tx}t^{q-1}\dd t$, $x > 0$, we have
\begin{equation}\label{eq:ident}
\E\left(\sum_{j=1}^n b_jY_j\right)^{-q} = \Gamma(q)^{-1}\int_0^\infty \exp\left(\sum_{j=1}^n \Lambda(tb_j)\right)t^{q-1}\dd t,
\end{equation}
where $\Lambda(u) = \log\E e^{-uY}$. We apply Lemma \ref{lm:st-exp} to the numbers $tb_j$ which add up to $t$. It is clear that under our assumptions on $Y$, the constant $c$ from Lemma \ref{lm:st-exp} satisfies $c \geq c_1e^{-c_2t}$, for some positive constants $c_1, c_2 > 0$ which depend only on the distribution of $Y$. Thus, from \eqref{eq:st-exp-min}, we get
\[
\E\left(\sum_{j=1}^n b_jY_j\right)^{-q} \geq \Gamma(q)^{-1}\int_0^\infty \exp\left(n\Lambda(t/n)+c_1e^{-c_2t}t^2\delta\right)t^{q-1}\dd t
\]
with 
$
\delta = \sum_{j=1}^n (b_j - 1/n)^2.
$
Using $\exp\left(c_1e^{-c_2t}t^2\delta\right) \geq c_1e^{-c_2t}t^2\delta + 1$, we obtain
\[
\E\left(\sum_{j=1}^n b_jY_j\right)^{-q} \geq  \E\left(\sum_{j=1}^n \frac{1}{n}Y_j\right)^{-q}  + \delta\cdot c_1\Gamma(q)^{-1}\int_0^\infty \exp\left(n\Lambda(t/n)\right)e^{-c_2t}t^{q+1}\dd t.
\]
By the convexity of $\Lambda$, the sequence $(n\Lambda(t/n))_n$ is nonincreasing with the limit $-t\E Y$, hence
\[
\int_0^\infty \exp\left(n\Lambda(t/n)\right)e^{-c_2t}t^{q+1}\dd t \geq \int_0^\infty e^{-(c_2+\E Y)t}t^{q+1} \dd t,
\]
which gives \eqref{eq:st-mom-min}. 

%To obtain \eqref{eq:st-Bp-max}, we proceed similarly. From \eqref{eq:ident} and \eqref{eq:st-exp-max}, we get
%\begin{align*}
%\E\left(\sum_{j=1}^n b_jY_j\right)^{-q} &\leq  \Gamma(q)^{-1}\int_0^\infty \exp\left(\Lambda(t)-c_1e^{-c_2t}t^2\delta\right)t^{q-1}\dd t \\
%&= \E Y_1^{-q} +  \Gamma(q)^{-1}\int_0^\infty \exp\left(\Lambda(t)\right)\Big[\exp\left(-c_1e^{-c_2t}t^2\delta\right)-1\Big]t^{q-1}\dd t, 
%\end{align*}
%where now
%\[
%\delta = 1 - \sum_{j=1}^n b_j^2.
%\]
%Since $c_1e^{-c_2t}t^2\delta \leq c_3$ for every $t \geq 0$ for some constant $c_3$ which depends only on $c_1, c_2$, there is a constant $c_1' > 0$ such that for every $t \geq 0$, we have 
%\[
%\exp\left(-c_1e^{-c_2t}t^2\delta\right)-1 \leq -c_1'e^{-c_2t}t^2\delta.
%\]
%Moreover, $\exp\left(\Lambda(t)\right) = \E e^{-t Y_1} \geq e^{-t \E Y_1}$. Putting these bounds together gives 
%\[
%\E\left(\sum_{j=1}^n b_jY_j\right)^{-q} \leq \E Y_1^{-q} - \delta\cdot c_1'\Gamma(q)^{-1}\int_0^\infty e^{-(c_2+\E Y_1)t}t^{q+1} \dd t,
%\]
%finishing the proof of \eqref{eq:st-Bp-max}.
\end{proof}

We are ready to establish the desired stability results for minimal sections.

\begin{proof}[Proof of \eqref{eq:st-Bp-p<2-min}]
Let
\[
A_{n,p} = \E\left(\sum_{j=1}^n \frac{1}{n}\bar V_j\right)^{-1/2}.
\]
From \eqref{eq:Bp-sec-p<2} and \eqref{eq:st-mom-min} applied to the $\bar V_j$ and $q = \frac{1}{2}$, we have
\begin{align}
\label{eq:Anp-ratio}\frac{\vol_{n-1}(B_p^n \cap a^\perp)}{\vol_{n-1}(B_p^n \cap (\frac{1}{\sqrt{n}},\dots,\frac{1}{\sqrt{n}})^\perp)} &= \frac{1}{A_{n,p}}\E\left(\sum_{j=1}^n a_j\bar V_j\right)^{-1/2} \\\notag
&\geq 1 + \frac{c_p}{A_{n,p}}\sum_{j=1}^n (a_j^2-1/n)^2
\end{align}
with a positive constant $c_p$ which depends only on $p$ (through the distribution of $\bar V_1$). It remains to note that thanks to Schur-convexity, the sequence $A_{n,p}$ is nonincreasing, thus $A_{n,p} \leq A_{1,p} = \E \bar V_1^{-1/2} = 1$.
\end{proof}

\begin{remark}\label{rem:Anp}
The sequence $A_{n,p}$ is in fact bounded below as well, namely by
\[
\lim_{n\to\infty} A_{n,p} \geq \E\left[\lim_{n\to\infty} \left(\sum_{j=1}^n \frac{1}{n}\bar V_j\right)^{-1/2}\right] = \left(\E \bar V_1\right)^{-1/2}.
\]
Moreover, as $n\to\infty$, we have
\begin{equation}\label{eq:Anp-asymp}
A_{n,p} = c_0(p) + \frac{c_1(p)}{n} + O(n^{-3/2})
\end{equation}
for some constants $c_0(p), c_1(p)$ which depend only on $p$. This is justified by first noting that $A_{n,p} = g_n(0)$  where $g_n(x)$ is the density of $\frac{1}{\sqrt{n}}\sum_{j=1}^n Y_j$ (plug in $a=e_1$ in \eqref{eq:Anp-ratio} and recall Corollary \ref{cor:secBpn}) and then evoking the Edgeworth expansion for $g_n$ (see, e.g. Theorem 3.2 in \cite{Bob} and classical references therein).
%\begin{align}
%A_{n,p} &= (\E \bar V_1)^{-1/2}\E\left(1 + \frac{1}{n}\sum_{j=1}^n\frac{\bar V_j - \E \bar V_j}{\E \bar V_j}\right)^{-1/2} \notag\\
%&= (\E \bar V_1)^{-1/2}\left(1 + n^{-1}\frac{3}{8}\frac{\Var(\bar V_1)}{(\E \bar V_1)^2} + O(n^{-3/2})\right). \label{eq:Anp-asymp}
%\end{align}

\end{remark}

\begin{remark}\label{rem:opt-Bp-min-p<2}
The dependence on $\delta_n(a) = \sum_{j=1}^n (a_j^2-1/n)^2$ in \eqref{eq:st-Bp-p<2-min} modulo a constant factor is best possible, in the following two scenarios.

1) As $n\to \infty$, there are unit vectors $a$ in $\R^n$ with $\delta_n = \delta_n(a) \to 0$ such that the left hand side of \eqref{eq:st-Bp-p<2-min} is in fact of the order $1 + c(p)\cdot \delta_n + o(\delta_n)$. Consider $a = (\frac{1}{\sqrt{n-1}},\dots, \frac{1}{\sqrt{n-1}}, 0)$ in $\R^n$. Then $\delta_n = \delta_n(a) = (n-1)\left(\frac{1}{n-1}-\frac{1}{n}\right)^2 + \frac{1}{n^2} = \frac{1}{n^2}+O\left(\frac{1}{n^3}\right)$ and, using \eqref{eq:Anp-asymp},
\begin{align*}
\frac{\vol_{n-1}(B_p^n \cap (\frac{1}{\sqrt{n-1}},\dots, \frac{1}{\sqrt{n-1}}, 0)^\perp)}{\vol_{n-1}(B_p^n \cap (\frac{1}{\sqrt{n}},\dots,\frac{1}{\sqrt{n}})^\perp)} &= \frac{A_{n-1,p}}{A_{n,p}} = 1  + \frac{c(p)}{n^2} + O\left(\frac{1}{n^{5/2}}\right).
\end{align*}

2) For a fixed $n$, there are unit vectors $a$ in $\R^n$ with $\delta = \delta_n(a) \to 0$ such that the left hand side of \eqref{eq:st-Bp-p<2-min} is of the order $1 + c(p,n) \delta + o(\delta)$. For simplicity, let $n$ be a fixed even integer. Let $\e \to 0^+$ and consider 
\[a_\e = (\underbrace{\sqrt{\frac{1}{n}+\e},\dots, \sqrt{\frac{1}{n}+\e}}_{n/2},\underbrace{\sqrt{\frac{1}{n}-\e},\dots, \sqrt{\frac{1}{n}-\e}}_{n/2}).
\]
Then $\delta_\e = \delta_n(a_\e) = n\e^2$ and with 
\[
X = \bar V_1 + \dots + \bar V_{n/2}, \quad Y = \bar V_{n/2+1} + \dots + \bar V_{n},
\] 
which are i.i.d., we have
\begin{align*}
\frac{\vol_{n-1}(B_p^n \cap a_\e^\perp)}{\vol_{n-1}(B_p^n \cap (\frac{1}{\sqrt{n}},\dots,\frac{1}{\sqrt{n}})^\perp)} &= \frac{1}{A_{n,p}}\E\left(\frac{X+Y}{n} + \e(X-Y)\right)^{-1/2} \\
&= \frac{1}{A_{n,p}}\E\left[\left(\frac{X+Y}{n}\right)^{-1/2}\left(1 + \e n\frac{X-Y}{X+Y}\right)^{-1/2}\right].
\end{align*}
Since $\left|\e n\frac{X-Y}{X+Y}\right| \leq \e n < \frac{1}{2}$, for sufficiently small $\e$, using $(1 + x)^{-1/2} \leq 1 - \frac{1}{2}x + x^2$, $x > -\frac12$, we can thus upper bound the right hand side by
\begin{align*}
\frac{1}{A_{n,p}}\E\left[\left(\frac{X+Y}{n}\right)^{-1/2}\left(1 -\frac{1}{2}\e n\frac{X-Y}{X+Y} + \e^2n^2\left(\frac{X-Y}{X+Y}\right)^2\right)\right] = 1 + c(p,n)\e^2,
\end{align*}
where we use that $\left|\frac{X-Y}{X+Y}\right| \leq 1$ to guarantee the existence of the expectations involved and symmetry to conclude that term linear in $\e$ vanishes.
\end{remark}

\subsection{Case \texorpdfstring{$2 < p < \infty$}{p>2}}\label{sec:p>2}
Here we prove \eqref{eq:st-Bp-p>2}. We use the formula from Corollary \ref{cor:secBpn}, that for a unit vector $a \in \R^n$, we have
\[
\frac{\vol_{n-1}(B_p^n \cap a^\perp)}{\vol_{n-1}(B_p^{n-1})} = f_{a}(0),
\]
where $f_a$ is the density of $\sum_{j=1}^n a_jY_j$, $Y_1, Y_2, \dots$ are i.i.d. random variables, each with density $\exp(-\beta_p^p|x|^p)$, where $\beta_p = 2\Gamma(1+1/p)$.

\begin{lemma}\label{lm:int-ineq-p>2}
Let $2 < p < \infty$. For every $u_0 > 0$, there is $c > 0$ depending only on $u_0$ and $p$ such that for every $0 < u < u_0$, we have
\begin{equation}\label{eq:int-ineq-p>2}
(1+u)^{1/2}\int_{\R} \exp\left\{-\beta_p^pu^{p/2}\left|x\right|^p-\pi x^2\right\} \dd x \geq 1 + cu.
\end{equation}
\end{lemma}
\begin{proof}
Fix $2 < p < \infty$ and $u_0 > 0$. Using $\exp(-t) \geq 1- t$, we obtain
\[
\int_{\R} \exp\left\{-\beta_p^pu^{p/2}\left|x\right|^p-\pi x^2\right\} \dd x \geq 1 - A_pu^{p/2}
\]
with $A_p = \beta_p^p\int_{\R} |x|^pe^{-\pi x^2}\dd x$. Thus it is clearly possible to choose sufficiently small $u_1 > 0$ and $c > 0$ which depend only on $p$ such that \eqref{eq:int-ineq-p>2} holds for all $0 < u < u_1$. Moreover, a change of variables $x = u^{-1/2}y$ yields
\[
\int_{\R} \exp\left\{-\beta_p^pu^{p/2}\left|x\right|^p-\pi x^2\right\} \dd x = u^{-1/2}\E \exp\left\{-\pi u^{-1}Y^2\right\},
\]
where $Y$ is a random variable with density $\exp(-\beta_p^p|x|^p)$ which is \emph{more} peaked than a Gaussian random variable $G$ with density $\exp(-\pi x^2)$. Thus, for every $u > 0$,
\[
\int_{\R} \exp\left\{-\beta_p^pu^{p/2}\left|x\right|^p-\pi x^2\right\} \dd x > u^{-1/2}\E \exp\left\{-\pi u^{-1}G^2\right\} = (1+u)^{-1/2}.
\]
Thus, by continuity, the infimum of left hand side of \eqref{eq:int-ineq-p>2} over $u_1 < u < u_0$ is strictly larger than $1$. Decreasing $c$ if necessary allows to finish the argument.
\end{proof}

%\begin{theorem}\label{thm:st-Bp-p>2}
%Let $2 < p < \infty$. There is a positive constant $c_p$ which depends only on $p$ such that for every $n$ and every unit vector $a$ in $\R^n$ with $a_1 \geq \dots \geq a_n \geq 0$, we have
%\begin{equation}\label{eq:st-Bp-p>2}
%\frac{\vol_{n-1}(B_p^n \cap a^\perp)}{\vol_{n-1}(B_p^{n-1})} \geq 1  +  c_p|a - e_1|^2.
%\end{equation}
%\end{theorem}
\begin{proof}[Proof of \eqref{eq:st-Bp-p>2}]
We use different arguments, depending on whether the vector $a$ is close or not to the minimising one $e_1$. With hindsight, fix $\theta_p$ to be a positive sufficiently small constant which depends only on $p$ such that
\begin{equation}\label{eq:theta}
\begin{split}
(2\pi \E Y_1^2)^{-1/2}\exp(-0.28&\theta_p(\E|Y_1|^3)(\E Y_1^2)^{-5/2}) \\
& - (0.56\theta_p(\E|Y_1|^3)(\E Y_1^2)^{-3/2})^{1/2} > 1.
\end{split}
\end{equation}
Such a choice is possible since $2\pi \E Y_1^2 < 1$ for $p > 2$, as explained later in the proof.

\emph{Case 1: $a_1 > \theta_p$.} Here the starting point is a formula obtained from writing $f_a(0)$ as the convolution of the densities $\frac{1}{a_j}\exp(-\beta_p^p|x_j/a_j|^p)$ and changing the variables $y_j = x_j/a_j$, leading to
\[
f_a(0) = \frac{1}{a_1}\E\exp\left\{-\beta_p^p\left|\sum_{j=2}^n b_jY_j\right|^p\right\},
\]
with $b_j = \frac{a_j}{a_1}$. 
Let
\[
u = \sum_{j=2}^n b_j^2 = \frac{1-a_1^2}{a_1^2}.
\]
Note that our assumption  $a_1 \geq \theta_p$ is equivalent to $u \leq \theta_p^{-2}-1$. 
Since $Y_j$ is \emph{more} peaked than a Gaussian with density $\exp(-\pi x^2)$, we get
\[
\E\exp\left\{-\beta_p^p\left|\sum_{j=2}^nb_jY_j\right|^p\right\} \geq \int_{\R} \exp\left\{-\beta_p^p\left(\sum_{j=2}^n b_j^2\right)^{p/2}\left|x\right|^p-\pi x^2\right\} \dd x.
\]
Note that $\frac{1}{a_1} = \sqrt{1+u}$. Lemma \ref{lm:int-ineq-p>2} applied with $u_0 = \theta_p^{-2}-1$ thus yields
\[
f_a(0) \geq 1+c_pu = 1 + c_p\frac{1-a_1^2}{a_1^2} \geq 1 + c_p(1-a_1).
\]
with a positive constant $c_p$ which depends only on $p$.

\emph{Case 2: $a_1 \leq \theta_p$.} Since in this case
\[
\rho = \sum_{j=1}^n \E|a_jY_j|^3 \leq a_1\E|Y_1|^3\sum_{j=1}^n a_j^2 \leq \theta_p\E|Y_1|^3,
\]
we can use the Berry-Esseen theorem to argue that $f_a(0)$ is \emph{large}. Let 
\[
\sigma_p = (\E Y_1^2)^{1/2}.
\]
We have (see, e.g. \cite{T12} which provides the current best value of the numerical constant in the Berry-Esseen theorem),
\[
\sup_{x \in \R} \left|\p{\sum_{j=1}^n a_jY_j \leq x} - \p{Z_p \leq x}\right| \leq 0.56\sigma_p^{-3}\rho,
\]
where $Z_p$ is a Gaussian random variable with variance $\sigma_p$. Let $\phi_p$ denote the density of $Z_p$. Crucially, peakedness yields
\[
\phi_p(0) = \frac{1}{\sqrt{2\pi}\sigma_p} > \frac{1}{\sqrt{2\pi}\sigma_2} =  1,
\]
since $p > 2$. Thanks to the symmetry and monotonicity of the densities involved, in particular we obtain that for every $\delta > 0$,
\[
\delta f_a(0) \geq \int_0^\delta f_a(x) \dd x \geq \int_0^\delta \phi_p(x) \dd x - \e_p
\]
with $\e_p = 0.56\theta_p\sigma_p^{-3}\E|Y_1|^3$. Letting, say $\delta = \e_p^{1/2}$ and using $\delta^{-1}\int_0^\delta \phi_p(x)\dd x > \phi_p(\delta)=\phi_p(0)e^{-\delta^2/(2\sigma_p^2)}$, we see that $\theta_p$ chosen sufficiently small according to \eqref{eq:theta} guarantees that
\[
f_a(0) \geq \e_p^{-1/2}\int_0^{\e_p^{1/2}} \phi_p(x) \dd x - \e_p^{1/2} \geq  \phi_p(0)e^{-\e_p/(2\sigma_p^2)}- \e_p^{1/2} = 1 + c_p  
\]
with a positive constant $c_p$ which depends only on $p$. This gives $f_p(0) \geq 1 + c_p$, which finishes the proof. 
\end{proof}

\begin{remark}\label{rem:opt-Bp-p>2}
It can be seen again by taking vectors with exactly two nonzero coordinates that the dependence on $\delta(a) = |a-e_1|^2$ in \eqref{eq:st-Bp-p>2} modulo a constant factor is best possible. For instance, take $\e \to 0$ and consider $a_\e = (\sqrt{1-\e}, \sqrt{\e}, 0, \dots, 0)$. Then $\delta_\e = \delta(a_\e) = 2(1-\sqrt{1-\e}) = \e + O(\e^2)$ and 
\[
\frac{\vol_{n-1}(B_p^n \cap a_\e^\perp)}{\vol_{n-1}(B_p^{n-1})} = \left((1-\e)^{p/2} + \e^{p/2}\right)^{-1/p} = 1 + \frac{1}{2}\e + O(\e^{p/2}) = 1 + \frac{1}{2}\delta_\e + o(\delta_\e),
\]
since $p > 2$.
\end{remark}

\section{Conclusion}\label{sec:conclusion}

Our result of Theorem \ref{thm:p=1}  confirms the intuition that the (unknown) extremal subspaces for minimal-volume central sections of $B_p^n$, $0 < p < 2$, are conceivably \emph{as symmetric as possible}. Note that in the case of the corresponding question for maximal-volume sections and $p>2$, the situation is more delicate, at least for \emph{large} $p$, as suggested by Ball's results (even in the hyperplane case).

It has been elusive how to extend the arguments from Section \ref{sec:p=1,k=2} to other values of $p$ than $p=1$, or higher dimensions $k$ than $k=2$. We conjecture that when $k=2$, the minimising subspace $H$ is the same as in Theorem \ref{thm:p=1} for all $0 < p < 2$.

Theorem \ref{thm:stab} deals only with the case of hyperplane sections. It would be of interest to ask for corresponding stability results for lower dimensional sections. We believe that (at least some of) our methods are robust enough to yield satisfactory answers. Another challenging and intriguing question is that of a sharp dependence on $p$ of the constants $c_p$ in Theorem \ref{thm:stab}.

\end{document}